\documentclass[article]{amsart}
\usepackage{amssymb,amsfonts,amsmath,amsthm}
\usepackage[all,arc]{xy}
\usepackage{enumerate}
\usepackage{mathrsfs}
\usepackage[toc,page]{appendix}
\usepackage[left=3cm, right=3cm, bottom=3cm]{geometry}
\usepackage{graphicx}
\usepackage{tabularx}
\usepackage{multirow} % multirow in table
\usepackage{caption} % table caption with no number
\usepackage{url}
\usepackage{color}
\usepackage{tikz-cd}  % graph
\usepackage{mathdots} %\ddots, \vdots, \iddots
\usepackage{romannum} %Roman number
\usepackage{hyperref} %hyperlink for reference
\usepackage{float} %table follow paragraphp

\newtheorem{thm}{Theorem}[section]

\newtheorem*{thm*}{Theorem}
\newtheorem*{cor*}{Corollary}
\newtheorem*{prop*}{Proposition}
\newtheorem{cor}[thm]{Corollary}
\newtheorem{prop}[thm]{Proposition}
\newtheorem{lem}[thm]{Lemma}

\theoremstyle{definition}
\newtheorem{defn}[thm]{Definition}

\newtheorem{conv*}{Convention}
\newtheorem{exmp}[thm]{Example}

\newtheorem{notn}[thm]{Notation}
\newtheorem*{notn*}{Notation}

\theoremstyle{remark}
\newtheorem{rem}[thm]{Remark}

\newtheorem*{idea*}{Idea}

\newcommand{\Spec}{{\rm Spec}}

\makeatletter
\let\c@equation\c@thm
\makeatother
\numberwithin{thm}{section}
\numberwithin{equation}{section}

\title[Nonabelian Hodge Correspondence]{A Nonabelian Hodge Correspondence for Principal Bundles in Positive Characteristic}

\author{Mao Sheng, Hao Sun and Jianping Wang}

\begin{document}
\pagenumbering{arabic}
\maketitle
\begin{abstract}
In this paper, we prove a nonabelian Hodge correspondence for principal bundles on a smooth variety $X$ in positive characteristic, which generalizes the Ogus--Vologodsky correspondence for vector bundles. Then we extend the correspondence to parahoric torsors over a log pair $(X,D)$, where $D$ a reduced normal crossing divisor in $X$. As an intermediate step, we prove a correspondence between principal bundles on root stacks $\mathscr{X}$ and parahoric torsors on $(X,D)$, which generalizes the correspondence on curves given by Balaji--Seshadri to higher dimensional case.
\end{abstract}
	
\flushbottom
	
%\tableofcontents
%\newpage
	
\renewcommand{\thefootnote}{\fnsymbol{footnote}}
\footnotetext[1]{Key words: nonabelian Hodge correspondence, G-Higgs bundle, flat G-bundle}
\footnotetext[2]{MSC2020:  14C30,  14L15, 20G15}
	
\section{Introduction}

Ogus--Vologodsky established the nonabelian Hodge correspondence for $G={\rm GL}_n$ in positive characteristic \cite{OV07}. A natural question is how to generalize Ogus--Vologodsky's result to an arbitrary reductive group $G$. Chen--Zhu established a nonabelian Hodge correspondence for reductive groups on curves with the help of Hitchin morphisms \cite{CZ15}. For nilpotent objects with small exponent their correspondence coincides with the Ogus--Vologodsky correspondence. However, their approach seems to be hard to be generalized to higher dimensional varieties, because there is still a lack of numerous necessary properties of the Hitchin morphism (see \cite[Conjecture 3.2 and Conjecture 5.2]{CN20} for instance).

As remarked in \cite[Remark 2.4]{LSZ15}, the exponential twisting approach to the Ogus--Vologodsky correspondence should respect $G$-structure for any subgroup $G\subset \mathrm{GL}_n$. In this paper, we aim to work out this point in detail, which is going to be applied to construct many absolutely irreducible crystalline $G$-local systems in a subsequent work. As would be clear to the reader, the key point in this approach is the construction of exponential maps for reductive groups with good properties. Fortunately, this has been studied by Serre \cite{Ser94} and several other authors in a series work \cite{Sei00,Mcn02,Sob15b}. In particular, Sobaje proved its existence and uniqueness when the characteristic is good \cite[Theorem 3.4]{Sob15b}. These knowledge suffice for generalizing the result of Lan--Sheng--Zuo \cite{LSZ15} in the vector bundle case to principal $G$-bundles, where $G$ denotes for a connected split reductive group $G$ over a perfect field $k$ in positive characteristic $p$. In effect, we establish a correspondence between nilpotent $G$-Higgs bundles and nilpotent flat $G$-bundles of exponent $\leq p-1$ for a smooth variety over $k$ which is $W_2(k)$-liftable. Next, we generalize this correspondence to the special log setting $(X,D)$, where $X$ is a smooth variety and $D$ is a reduced normal crossing divisor of $X$. There has been the work by Schepler \cite{Sch05} and the work \cite{LSYZ19} via exponential twisting in the vector bundle case, and the work by Krishnamoorthy and the first named author \cite{KS20a} for parabolic bundles over curves. The right generalization to principal bundles of parabolic Higgs bundles which have simple poles along $D$ is the notion of logahoric (=logarithmic parahoric) Higgs torsors (see \cite{KSZ23} for instance). Like in \cite{KS20a}, we first establish a logarithmic version of the correspondence over a tame root stack. However, the procedure of descending logarithmic Higgs torsors/connections from root stacks to logahoric Higgs torsors/connections over $(X,D)$ is a bit different from the generalized Biswas-Iyer-Simpson correspondence for parabolic vector bundles in loc. cit. When $X$ is projective equipped with an ample line bundle $H$, one has the notion of $R$-stability \cite{Ram75,Ram78} for principal bundles which generalizes the Mumford-Takemoto slope stability in vector bundle case \cite[\S 3]{Ram75}. It has been further generalized to a log pair $(X,D)$ in \cite{KSZ23}. We show that our correspondence preserves the $R$-stability condition, as one may expect.

\hspace{\fill}

Now we give the setup and main results in this paper. Let $k$ be a perfect field in positive characteristic $p$. Let $X$ be a $W_2(k)$-liftable smooth variety over $k$ and denote by $X'$ the Frobenius pullback. Let $G$ be a connected split reductive linear algebraic group over $k$ and we suppose that the characteristic $p$ is good \cite[\S 2.9]{Her13}. We introduce the following categories
\begin{itemize}
    \item ${\rm HIG}_{p-1}(X'/k,G)$: the category of nilpotent $G$-Higgs bundles of exponent $\leq p-1$ on $X'$;
    \item ${\rm MIC}_{p-1}(X/k,G)$: the category of nilpotent flat $G$-bundles of exponent $\leq p-1$ on $X$,
\end{itemize}
and prove that they are equivalent by establishing the Cartier and inverse Cartier transform via exponential twisting in \S\ref{sect_HIG_MIC_G}.
\begin{thm}[Theorem \ref{thm_HIG_MIC_G}]
Suppose that $X$ is $W_2(k)$-liftable. The categories ${\rm HIG}_{p-1}(X'/k,G)$ and ${\rm MIC}_{p-1}(X/k,G)$ are equivalent.
\end{thm}
\noindent Suppose $X$ is projective and we fix an ample line bundle $H$ on $X$. We follow Ramanathan's approach \cite{Ram75,Ram96a} to define the stability condition of $G$-bundles on $X$ of an arbitrary dimension, which is called \emph{$R$-stability condition} (see \S\ref{subsect_G_Higgs_flat}). Denote by ${\rm HIG}^{s(s)}_{p-1}(X'/k,G) \subseteq {\rm HIG}_{p-1}(X'/k)$ (resp. ${\rm MIC}^{s(s)}_{p-1}(X/k,G) \subseteq {\rm MIC}_{p-1}(X/k)$) the full subcategory of $R$-(semi)stable nilpotent $G$-Higgs bundles (resp. flat $G$-bundles) of exponent $\leq p-1$. We prove that the correspondence also holds under $R$-stability conditions (Theorem \ref{thm_HIG_MIC_G_stab}).

Now we consider the correspondence on $X$ together with a reduced normal crossing divisor $D=\sum_{i=1}^s D_i$, where $D_i$ are irreducible components for $1 \leq i \leq s$. By introducing parahoric torsors, we prove a logarithmic version of the nonabelian Hodge correspondence. In a nutshell, we prove the equivalence of categories between nilpotent logahoric Higgs torsors and nilpotent logahoric connections of exponent $\leq p-1$. Here is a sketch of the proof. We equip each $D_i$ with a tame weight $\alpha_i$, where a tame weight is a cocharacter with rational coefficients such that the denominator $d_i$ is coprime to $p$. Denote by $\boldsymbol\alpha$ (resp. $\boldsymbol{d}$) the collection of weights (resp. integers). With respect to the above data, there exists a (tame) root stack $\mathscr{X}$ (see \cite[Theorem 4.1]{MO05} for example). We define the following categories
\begin{itemize}
    \item ${\rm HIG}_{p-1}(\mathscr{X}'_{\rm log}/k,G)$: the category of nilpotent logarithmic $G$-Higgs bundles of exponent $\leq p-1$ on $\mathscr{X}'$;
    \item ${\rm MIC}_{p-1}(\mathscr{X}_{\rm log}/k,G)$: the category of nilpotent logarithmic flat $G$-bundles of exponent $\leq p-1$ on $\mathscr{X}$, of which the residues are also nilpotent of exponent $\leq p-1$.
\end{itemize}
We first prove a logarithmic version of the nonabelian Hodge correspondence on the root stack $\mathscr{X}$ by applying results in \S\ref{sect_HIG_MIC_G} and \cite[Appendix]{LSYZ19}.
\begin{thm}[Theorem \ref{thm_HIG_MIC_G_log_stack}]
Suppose that $(\mathscr{X},\widetilde{D})$ is $W_2$-liftable. The categories ${\rm HIG}_{p-1}(\mathscr{X}'_{\rm log}/k , G)$ and ${\rm MIC}_{p-1}(\mathscr{X}_{\rm log}/k , G)$ are equivalent.
\end{thm}

Then let $\mathscr{F}$ be a $G$-bundle on $\mathscr{X}$. We can construct a parahoric group scheme $\mathcal{G}_{\boldsymbol\alpha}$ by Lemma \ref{lem_parah_grp}. Note that this parahoric group scheme $\mathcal{G}_{\boldsymbol\alpha}$ depends on the given $G$-bundle $\mathscr{F}$ and the weights $\boldsymbol\alpha$ is determined by the type of $\mathscr{F}$. Similar to the case of curves \cite[\S 3]{BS15}, we prove that parahoric $\mathcal{G}_{\boldsymbol\alpha}$-torsors on $X$ are in one-to-one correpsondence with $G$-bundles on $\mathscr{X}$ of type $\boldsymbol\rho$, where $\boldsymbol\rho = \{\rho_i, 1 \leq i \leq s\}$ is a collection of representations $\rho_i : \mu_{d_i} \rightarrow T$ determined by $\boldsymbol\alpha$ and $T$ is a maximal torus of $G$.

\begin{prop}[Proposition \ref{prop_parah_equiv_G}]
    The category of parahoric $\mathcal{G}_{\boldsymbol\alpha}$-torsors on $X$ and the category of $G$-bundles on $\mathscr{X}$ of type $\boldsymbol\rho$ are equivalent.
\end{prop}

This correspondence carries over $G$-bundles with Higgs field/connection. It has been done in \cite{KSZ23} for curves in characteristic zero, and in \cite{KS20a} for vector bundles over curves in positive characteristic. We first introduce some definitions. A \emph{logahoric $\mathcal{G}_{\boldsymbol\alpha}$-Higgs torsor} is a pair $(\mathcal{E},\vartheta)$, where $\mathcal{E}$ is a parahoric $\mathcal{G}_{\boldsymbol\alpha}$-torsor and $\vartheta : X \rightarrow \mathcal{E}(\mathfrak{g}) \otimes \Omega_X({\rm log}\, D )$ is a logarithmic Higgs field. A \emph{logahoric $\mathcal{G}_{\boldsymbol\alpha}$-connection} is a pair $(\mathcal{V},\nabla')$, where $\mathcal{V}$ is a parahoric $\mathcal{G}_{\boldsymbol\alpha}$-torsor and $\nabla': \mathcal{O}_{\mathcal{V}} \rightarrow \mathcal{O}_{\mathcal{V}} \otimes \Omega_X( {\rm log} \, D )$ is a logarithmic integrable $\mathcal{G}_{\boldsymbol\alpha}$-connection. We define
\begin{itemize}
    \item ${\rm HIG}_{p-1}(X_{\rm log}/k, \mathcal{G}_{\boldsymbol\alpha})$: the category of nilpotent logarhoic $\mathcal{G}_{\boldsymbol\alpha}$-Higgs torsors of exponent $\leq p-1$ on $X$;
    \item ${\rm MIC}_{p-1}(X_{\rm log}/k, \mathcal{G}_{\boldsymbol\alpha})$: the category of nilpotent logahoric $\mathcal{G}_{\boldsymbol\alpha}$-connections of exponent $\leq p-1$ on $X$ such that the semisimple parts of residues are $\boldsymbol\alpha$.
\end{itemize}
Note that, in the above definition, the residue of a logarithmic connection on a parahoric torsor is not necessarily to be nilpotent and the semisimple part is given by the weights $\boldsymbol\alpha$.

\begin{prop}[Corollary \ref{cor_parah_equiv_G_Higgs} and Proposition \ref{prop_parah_equiv_G_conn}]
    The categories
    \begin{itemize}
        \item ${\rm HIG}_{p-1}(X_{\rm log}/k, \mathcal{G}_{\boldsymbol\alpha})$ and ${\rm HIG}_{p-1}(\mathscr{X}_{\rm log}/k,G,\boldsymbol\rho)$,
        \item ${\rm MIC}_{p-1}(X_{\rm log}/k, \mathcal{G}_{\boldsymbol\alpha})$ and ${\rm MIC}_{p-1}(\mathscr{X}_{\rm log}/k,G,\boldsymbol\rho)$
    \end{itemize}
    are equivalent.
\end{prop}
\noindent Furthermore, the above equivalences also hold under $R$-stability conditions for $G$-Higgs bundles and flat $G$-bundles. Therefore, we follow the diagram below and  prove a logahoric version of the nonabelian Hodge correspondence:
\begin{equation*}
\begin{tikzcd}
{\rm HIG}_{p-1}(\mathscr{X}'_{\rm log}/k , G, \boldsymbol\rho) \arrow[rr, "{\rm Theorem} \,\ref{thm_HIG_MIC_G_log_stack}", equal] \arrow[dd, "{\rm Corollary} \, \ref{cor_parah_equiv_G_Higgs}", equal] & & {\rm MIC}_{p-1}(\mathscr{X}_{\rm log}/k, G ,\boldsymbol\rho^p) \arrow[dd,"{\rm Proposition} \, \ref{prop_parah_equiv_G_conn}", equal] \\
& & \\
{\rm HIG}_{p-1}(X'_{\rm log}/k,\mathcal{G}_{\boldsymbol\alpha}) \arrow[rr, dotted, dash] & & {\rm MIC}_{p-1}(X_{\rm log}/k, \mathcal{G}_{\boldsymbol\alpha'})
\end{tikzcd}
\end{equation*}
where $\boldsymbol\rho^p = \{\rho_i^{p}, 1 \leq i \leq s\}$ and $\boldsymbol\alpha'$ is a collection of weights determined by $\boldsymbol\rho^p$.

\begin{thm}[Theorem \ref{thm_log_NAHC}]\label{thm_main_result_log}
Suppose that $(X,D)$ is $W_2(k)$-liftable. The categories ${\rm HIG}_{p-1}(X'_{\rm log}/k,\mathcal{G}_{\boldsymbol\alpha})$ and ${\rm MIC}_{p-1}(X_{\rm log}/k, \mathcal{G}_{\boldsymbol\alpha'})$ are equivalent. Moreover, the equivalence preserves stability conditions.
\end{thm}

\hfill{\space}

We make some further remarks and explain relations to previous work as follows.

Ogus--Vologodsky established the nonabelian Hodge correspondence in positive characteristic for $G={\rm GL}_n$ \cite{OV07} and the logarithmic case was studied by Schepler \cite{Sch05}. Our work is a generalization of Ogus--Vologodsky's correspondence for reductive groups and we also extend the results to parahoric torsors. Also, Chen--Zhu introduced an analogue of Hitchin's equations for principal bundles in positive characteristic
\begin{align*}
\begin{cases}
    \psi(\nabla - \theta) = 0 \\
    ((\nabla - \theta) \otimes \nabla^{can})(\psi(\nabla)) = 0 \ ,
\end{cases}
\end{align*}
where $\nabla$ is a connection, $\theta$ is a Higgs field and $\psi$ is the $p$-curvature \cite[Introduction]{CZ15}. The main results in this paper (Theorem \ref{thm_HIG_MIC_G} and Theorem \ref{thm_log_NAHC}) provide solutions for these equations in higher dimensional case.

Parahoric torsors on curves were studied by many groups of people. Heinloth studied uniformization property of parahoric torsors \cite{Hei10} and Yun studied parahoric Higgs torsors in \cite{Yun11}. After that,  Balaji--Seshadri constructed the moduli space of parahoric torsors in the case of curves \cite{BS15}. For higher dimensional varieties, we give the definition of parahoric group schemes (Definition \ref{defn_parah_grp_sch}), which is inspired by Balaji--Pandey's result \cite[Theorem 5.4]{BP22}, and prove that parahoric torsors are equivalent to principal bundles on an appropriate root stack (Proposition \ref{prop_parah_equiv_G}).

The $R$-stability condition for $G$-bundles is introduced by Ramanathan in the case of curves \cite{Ram75,Ram96a}, and studied further in higher dimensional case \cite{AB01,GLSS08}. When $G$ is a classical group, the $R$-stability condition of a principal bundle is equivalent to the slope stability condition of the corresponding associated bundle in the case of curves. We refer the reader to \cite[\S 6]{CS21} and \cite{Yan22} for more details. In a similar way, Heinloth define a stability condition for parahoric torsors \cite[\S 1]{Hei17}. The stability condition he considered is equivalent to the slope stability condition, which does not include the information of weights. On the other hand, the stability condition considered for parahoric torsors in this paper is an analogue of that given by Balaji--Seshadri \cite[\S 6]{BS15} and includes the information of weights (Definition \ref{defn_stab_cond_parah}), which is equivalent to the stability condition of the corresponding parabolic bundles when $G={\rm GL}_n$.

In the logarithmic case, the residues of nilpotent logarithmic connections are not nilpotent in general and the semisimple part of the residue is given by the corresponding weight. This phenomenon actually comes from the establishment of the equivalence between parahoric torsors on $X$ and principal bundles on the root stack $\mathscr{X}$. In positive characteristic, this phenomenon first appears in \cite{KS20a}, where the authors studied the case of $G={\rm GL}_n$ and named the condition as \emph{adjustedness} \cite[Definition 2.9]{KS20a}. Later on, the authors generalized the result to parahoric torsors on curves \cite{LS21}, which also generalized Chen--Zhu's result \cite{CZ15}. Moreover, in characteristic zero, a similar discussion is also given in \cite[\S 5]{Sim90}.

\section{Preliminaries}\label{sect_prel}

\subsection{G-Higgs bundles and flat G-bundles}\label{subsect_G_Higgs_flat}

Let $X$ be a smooth variety over a perfect field $k$ in positive characteristic $p$. Denote by $T_{X/k}$ (resp. $\Omega_{X/k}$) the relative tangent (resp. cotangent) sheaf. If there is no ambiguity, we use the notation $T_X$ and $\Omega_X$ for simplicity. We have the following well-known diagram about the Frobenius morphism:
\begin{center}
\begin{tikzcd}
X \arrow[rr, "F_{X/k}" description] \arrow[rrd] \arrow[rrrr, bend left, "F_X" description] & & X' \arrow[d] \arrow[rr, "\pi_{X/k}" description] & & X \arrow[d]\\
& & \Spec \, k \arrow[rr, "F_k" description] & & \Spec \, k
\end{tikzcd}
\end{center}
where
\begin{itemize}
    \item $X':= X \times_{\Spec \, k , F_k} \Spec \, k$,
    \item $F_k$ (resp. $F_X$) is the absolute Frobenius morphism of $\Spec \, k$ (resp. $X$),
    \item $F_{X/k}$ is the relative Frobenius morphism.
\end{itemize}
Now let $G$ be a connected split reductive linear algebraic group over $k$ with maximal torus $T$. Let $\mathfrak{g}$ (resp. $\mathfrak{t}$) be the Lie algebra of $G$ (resp. $T$). Let $E$ be a $G$-bundle on $X$ with adjoint bundle $E(\mathfrak{g}):=E \times_{G} \mathfrak{g}$. We have a short exact sequence called the \emph{Atiyah sequence}
\begin{align*}
    0 \rightarrow E(\mathfrak{g}) \rightarrow {\rm At}(E) \xrightarrow{q} T_X \rightarrow 0,
\end{align*}
where ${\rm At}(E)$ is the Atiyah algebroid (see \cite[A.2]{CZ15} or \cite{Bis10} for instance).

\begin{defn}\label{defn_lambda_connection}
    Let $\lambda$ be an element in $k$. A \emph{$\lambda$-connection} on $E$ is a $\lambda$-twisted splitting of the Atiyah sequence, that is a morphism $\nabla: T_X \rightarrow {\rm At}(E)$ such that $q \circ \nabla =  \lambda \cdot {\rm id}$. A $\lambda$-connection $\nabla$ is \emph{integrable} if $[\nabla(u),\nabla(v)] = \lambda \nabla([u,v])$ for any $u,v \in T_X$. Such a pair $(E,\nabla)$ is also called a \emph{$\lambda$-connection} for convenience. The \emph{$p$-curvature} $\psi_\nabla: T_X \rightarrow {\rm At}(E)$ of a $\lambda$-connection $\nabla$ is defined by the formula
    \begin{align*}
        \psi_{\nabla} (v) = (\nabla(v))^p - \lambda^{p-1} \nabla (v^{[p]})
    \end{align*}
    for $v \in T_X$. If there is no ambiguity, we use the notation $\psi := \psi_\nabla$ for simplicity.
\end{defn}

\begin{rem}
The $p$-curvature $\psi$ of a $\lambda$-connection is $p$-linear and satisfies $q \circ \psi = 0$. Then $\psi$ factor through $E(\mathfrak{g})$ and thus can be regarded as a section $E(\mathfrak{g}) \otimes F^*_{X/k} \Omega_{X'}$. A similar discussion was given in \cite[\S 5]{HZ23}. Moreover, a $\lambda$-connection $\nabla$ can be regarded as a map $T_U \rightarrow E|_U(\mathfrak{g})$ locally on affine open subset $U$.
\end{rem}

\begin{rem}\label{rem_Higgs_conn}
In this paper, two cases are of most concern: $0$-connections ($G$-Higgs bundles) and $1$-connections (flat $G$-bundles).

A $0$-connection $\theta$ satisfies $q \circ \theta = 0$. Then $\theta$ factors through $E(\mathfrak{g})$. If $\theta$ is integrable, we obtain a Higgs field $\theta: \mathcal{O}_X \rightarrow E(\mathfrak{g}) \otimes \Omega_X$. Therefore, such a pair $(E,\theta)$ is exactly a $G$-Higgs bundle.

A $1$-connection $\nabla$ is exactly a connection. If $\nabla$ is integrable, the pair $(V,\nabla)$ is a flat $G$-bundle. Furthermore, there is an equivalent definition of connections on $G$-bundles, which is also frequently used in this paper. We give it here for the convenience of the reader and refer to \cite[A.1]{CZ15} for more details. There is a canonical connection on $\mathcal{O}_{G \times X}$ and denote it by $\nabla_G$. A(n) \emph{(integrable) $G$-connection} $\nabla$ on $E$ is a(n) (integrable) connection $\nabla: \mathcal{O}_E \rightarrow \mathcal{O}_E \otimes_{\mathcal{O}_X} \Omega_{X}$ compatible with the multiplication of $\mathcal{O}_E$ such that the following diagram commutes
\begin{equation*}
	\begin{tikzcd}
	\mathcal{O}_E \arrow[r, "a"] \arrow[d, "\nabla"] & \mathcal{O}_E \otimes_{\mathcal{O}_X} \mathcal{O}_{G \times X} \arrow[d,"\nabla \otimes 1 + 1 \otimes \nabla_G"] \\
	\mathcal{O}_E \otimes_{\mathcal{O}_X} \Omega_{X} \arrow[r, "a \otimes 1"] & (\mathcal{O}_E \otimes_{\mathcal{O}_X} \mathcal{O}_{G \times X}) \otimes_{\mathcal{O}_X} \Omega_{X}
	\end{tikzcd}
\end{equation*}
where $a$ is the co-action map. Moreover, for each $v \in F^*_{X/k} T_{X'}$, the $p$-curvature $\psi$ of $\nabla$ induces the following commutative diagram
\begin{equation*}
	\begin{tikzcd}
	\mathcal{O}_E \arrow[r, "a"] \arrow[d, "\psi(v)"] & \mathcal{O}_E \otimes_{\mathcal{O}_X} \mathcal{O}_{G \times X} \arrow[d,"\psi(v) \otimes 1"] \\
	\mathcal{O}_E \arrow[r, "a"] & \mathcal{O}_E \otimes_{\mathcal{O}_X} \mathcal{O}_{G \times X}
	\end{tikzcd}
\end{equation*}
and we refer the reader to \cite[A.6]{CZ15} for more details.
\end{rem}

In the following, we consider stability conditions and suppose that $X$ is a smooth projective variety. We fix an ample line bundle $H$ on $X$. The \emph{degree} of a vector bundle $E$ on $X$ is given with respect to $H$ and denote it by $\deg_H E$. If there is no ambiguity, we use the notation $\deg E$ and omit the subscript $H$. We briefly review the stability condition for principal bundles on higher dimensional varieties (see \cite[Definition 3]{Ram78}, \cite[Definition 1.1]{AB01} and \cite[\S 3.2]{GLSS08} for instance), which is first considered by Ramanathan \cite{Ram75,Ram96a,Ram96b}. In this paper, this stability condition is called the \emph{$R$-stability condition}.

Let $P$ be a parabolic subgroup of $G$ together with a reduction of structure group $\sigma: X^{\rm big} \rightarrow (E|_{X^{\rm big}})/P$, where $X^{\rm big} \subseteq X$ is a big open subset. In this paper, a big open subset $X^{\rm big} \subseteq X$ means that ${\rm codim}(X \backslash X^{\rm big}) \geq 2$. We get a $P$-bundle $E_\sigma$ in the following way
\begin{center}
	\begin{tikzcd}
	(E|_{X^{\rm big}})_\sigma \arrow[r, dotted] \arrow[d, dotted] & E|_{X^{\rm big}} \arrow[d] \\
	X^{\rm big} \arrow[r,"\sigma"] & (E|_{X^{\rm big}})/P
	\end{tikzcd}
\end{center}
Given a character $\chi: P \rightarrow \mathbb{G}_m$, we obtain a line bundle $\chi_* ((E|_{X^{\rm big}})_\sigma)$ on $X^{\rm big}$ and denote it by $L(\sigma,\chi)$.
	
\begin{defn}[$R$-stability condition]\label{defn_R_stab_G}
	A $G$-bundle $E$ is \emph{$R$-semistable} (resp. \emph{$R$-stable}), if
	\begin{itemize}
	\item for any maximal parabolic subgroup $P \subseteq G$;
	\item for any antidominant character $\chi: P \rightarrow \mathbb{G}_m$ trivial on the center;
    \item for any big open subset $X^{\rm big} \subseteq X$;
	\item for any reduction of structure group $\sigma: X^{\rm big} \rightarrow (E|
_{X^{\rm big}})/P$,
\end{itemize}
we have
\begin{align*}
\deg L(\sigma,\chi) \geq 0 \text{ (resp. $> 0$)}.
\end{align*}
\end{defn}

Now we consider the $R$-stability condition of $\lambda$-connections $(E,\nabla)$. To simplify the notation, we give the following definition for reductions of structure group on $X$ rather than on $X^{\rm big}$. Let $\sigma: X \rightarrow E/P$ be a reduction of structure group. We obtain a $P$-bundle $E_\sigma$ and two natural maps
\begin{align*}
	 E_\sigma(\mathfrak{p}) \rightarrow E({\mathfrak{g}}), \quad {\rm At}(E_\sigma) \rightarrow {\rm At}(E).
\end{align*}
The reduction of structure group $\sigma$ is said to be \emph{$\nabla$-compatible} if there is a $\lambda$-connection $\nabla_\sigma: T_X \rightarrow {\rm At}(E_\sigma)$ such that the diagram commutes
\begin{center}
\begin{tikzcd}
	{\rm At}(E) & T_X \arrow[l, "\nabla" description] \arrow[ld, "\nabla_\sigma" description, dotted] \\
	{\rm At}(E_\sigma) \arrow[u] \ .&
\end{tikzcd}
\end{center}

\begin{defn}\label{defn_stab_lambda}
A $\lambda$-connection $(E,\nabla)$ is \emph{$R$-semistable} (resp. \emph{$R$-stable}), if
\begin{itemize}
	\item for any maximal parabolic subgroup $P \subseteq G$;
	\item for any antidominant character $\chi: P \rightarrow \mathbb{G}_m$ trivial on the center;
    \item for any big open subset $X^{\rm big} \subseteq X$;
	\item for any $\nabla$-compatible reduction of structure group $\sigma: X^{\rm big} \rightarrow (E|
_{X^{\rm big}})/P$,
\end{itemize}
we have
\begin{align*}
    \deg L(\sigma,\chi) \geq 0 \text{ (resp. $> 0$)}.
\end{align*}
\end{defn}

\begin{rem}
We give an equivalent definition of the $R$-stability condition in the case of $G$-Higgs bundles ($0$-connections). Let $(E,\theta)$ be a $G$-Higgs bundle. A reduction of structure group $\sigma$ is said to be \emph{$\theta$-compatible} if there is a lifting $\theta_\sigma: X \rightarrow E_{\sigma}(\mathfrak{p}) \otimes \Omega_{X}$ such that the diagram commutes
\begin{center}
\begin{tikzcd}
	& E_\sigma(\mathfrak{p}) \otimes \Omega_{X} \arrow[d] \\
		\mathcal{O}_X \arrow[r,"\theta"] \arrow[ur,"\theta_\sigma", dotted] & E(\mathfrak{g}) \otimes \Omega_{X}.
\end{tikzcd}
\end{center}

A $G$-Higgs bundle $(E,\theta)$ is \emph{$R$-semistable} (resp. \emph{$R$-stable}), if
\begin{itemize}
	\item for any maximal parabolic subgroup $P \subseteq G$;
	\item for any antidominant character $\chi: P \rightarrow \mathbb{G}_m$ trivial on the center;
    \item for any big open subset $X^{\rm big} \subseteq X$;
	\item for any $\nabla$-compatible reduction of structure group $\sigma: X^{\rm big} \rightarrow (E|
_{X^{\rm big}})/P$,
\end{itemize}
we have
\begin{align*}
    \deg L(\sigma,\chi) \geq 0 \text{ (resp. $> 0$)}.
\end{align*}

By Remark \ref{rem_Higgs_conn}, when $\lambda=0$, a $\nabla$-compatible reduction is equivalent to a $\theta$-compatible reduction. Therefore, the above definition is equivalent to Definition \ref{defn_stab_lambda}.
\end{rem}

\subsection{Nilpotency}\label{subsect_nil}

We first review the nilpotency for Higgs bundles and flat bundles.

\begin{defn}[{\cite[P.860]{LSZ15}}]\label{claasical}
A Higgs bundle $(E,\theta)$ (resp. flat bundle $(V,\nabla)$)  is said to be \emph{nilpotent of exponent $\leq n$} if there exists a covering of $X$ by open affine subsets $U_i$ such that for each $U_i$ and arbitrary sections $\partial_1,\dots,\partial_{n+1}$ of $T_{U_i}$ (resp. $F^*_{U'_i/k} T_{U'_i}$), we have
$$\theta|_{U_i}(\partial_1)\cdots\theta|_{U_i}(\partial_{n+1})=0\ (\text{resp.}\ \psi|_{U_i}(\partial_1) \cdots  \psi|_{U_i}(\partial_{n+1}) =0).$$
\end{defn}

Here is an equivalent description of the nilpotency.

\begin{prop}\label{equivalent}
A rank $r$ Higgs bundle $(E,\theta)$ (resp. flat bundle $(V,\nabla)$) is nilpotent of exponent $\leq n$ if and only if there exists a covering of $X$ by affine open subsets $U_i$ such that for each $U_i$, the set
$$\{\theta|_{U_i} (\partial)\ |\ \partial\ \text{is a section of }T_{U_i} \} \quad (resp. \text{ } \{\psi|_{U_i} (\partial)\ |\ \partial\ \text{is a section of }F^*_{U'_i/k}T_{U'_i} \})$$ lies in $\mathfrak{u}(\mathcal{O}_{U_i})$, where $\mathfrak{u}$ is the Lie algebra of a unipotent radical $U$ of a parabolic subgroup $P$ of ${\rm GL}_r$ such that the nilpotency class of $U$ is smaller than $n+1$.
\end{prop}

The proof of this proposition follows directly from the following result in linear algebra.

\begin{lem}\label{lem_linear_alg}
Let $\{A_1,\dots,A_m\}$ be a set of pairwise commutative nilpotent matrices in $\mathfrak{gl}_r$. The equality
$$A_1^{k_1}\cdots A_m^{k_m}=0$$
hold for any nonnegative integers $k_1,\dots, k_m$ such that $k_1+\dots +k_m=n+1$ if and only if $\{A_1,\dots,A_m\}$ lies in the Lie algebra $\mathfrak{u}$ of a nilpotent group $U$ with nilpotency class smaller than $n+1$, where $U$ is the unipotent radical of a parabolic subgroup $P$ of ${\rm GL}_r$.
\end{lem}

\begin{proof}
We first assume that $\{A_1,\dots,A_m\}$ lies in such a Lie algebra $\mathfrak{u}$. Without loss of generality, we may assume that the parabolic subgroup $P$ is in the standard form, that is, elements of $P$ are of
\begin{equation*}
    \begin{pmatrix}
    A_{11}&A_{12}&A_{13}&\cdots&A_{1k}\\
    0&A_{22}&A_{23}&\cdots&A_{2k}\\
    \vdots& \vdots& \vdots&\dots&\vdots\\
    0&0&0&\cdots&A_{kk}\\
    \end{pmatrix},
\end{equation*}
where $A_{ij}$ is a $d_i\times d_j$ block. The unipotent radical $U$ of $P$ is
\begin{equation*}
    \begin{pmatrix}
    I_{d_1}&A_{12}&A_{13}&\cdots&A_{1k}\\
    0& I_{d_2}&A_{23}&\cdots&A_{2k}\\
    \vdots& \vdots& \vdots&\dots&\vdots\\
    0&0&0&\cdots& I_{d_k}\\
    \end{pmatrix},
\end{equation*}
together with its Lie algebra $\mathfrak{u}$
\begin{equation*}
    \begin{pmatrix}
    0&A_{12}&A_{13}&\cdots&A_{1k}\\
    0& 0&A_{23}&\cdots&A_{2k}\\
    \vdots& \vdots& \vdots&\dots&\vdots\\
    0&0&0&\cdots& 0\\
    \end{pmatrix}.
\end{equation*}
Then the equality
$$A_1^{k_1}\cdots A_m^{k_m}=0$$
holds for any nonnegative integers $k_1,\dots, k_m$ such that $k_1+\dots +k_m\geq k$ since $A_i\in \mathfrak{u}$. Since the nilpotency class of $U$ is smaller than $n+1$, we have $k \leq n+1$. Thus, the equality
$$A_1^{k_1}\cdots A_m^{k_m}=0$$
holds for any nonnegative integers $k_1,\dots, k_m$ such that $k_1+\dots +k_m=n+1$.

For the other direction, let  $\{A_1,\dots,A_m\}$ be a set of pairwise commutative nilpotent matrices such that the equalities $$A_1^{k_1}\cdots A_m^{k_m}=0$$ hold for all nonnegative integers $k_1,\dots, k_m$ such that $k_1+\dots +k_m=n+1$. We need to find a parabolic subgroup $P$ of ${\rm GL}_r$ with unipotent radical $U$, of which the nilpotency class smaller than $n+1$. Indeed, this was proved in \cite[P.165-166]{Kos93} and we include the proof here for completeness. Consider each $A_i$ as an endmorphism of $V=k^r$, and define subspaces
$$V_i:=\mathop{\bigcap}\limits_{k_1+\dots +k_m=i}\ker (A_1^{k_1}\cdots A_m^{k_m})$$
for $0\leq i\leq n+1$. Then we obtain a flag of $V$:
$$0=V_0\subset V_1\subset\cdots\subset V_{n+1}=V.$$
We can choose a basis of $V$ as
$$\{\alpha_{11},\dots, \alpha_{1d_1};\ \alpha_{21},\dots, \alpha_{2d_2};\ \cdots ;\ \alpha_{n+1,1},\dots, \alpha_{n+1,d_{n+1}} \}$$
such that
$$\{\alpha_{11},\dots, \alpha_{1d_1};\ \alpha_{21},\dots, \alpha_{2d_2};\ \cdots ;\ \alpha_{i,1},\dots, \alpha_{i,d_i} \}$$ is a basis of $V_i$ for each $i$. Under this basis, each matrix $A_i$ has the form
\begin{equation*}
    \begin{pmatrix}
    0&A_{12}&A_{13}&\cdots&A_{1,n+1}\\
    0& 0&A_{23}&\cdots&A_{2,n+1}\\
    \vdots& \vdots& \vdots&\dots&\vdots\\
    0&0&0&\cdots& 0\\
    \end{pmatrix},
\end{equation*}
Now we define the parabolic subgroup $P$ to be
\begin{equation*}
    \begin{pmatrix}
    A_{11}&A_{12}&A_{13}&\cdots&A_{1,n+1}\\
    0&A_{22}&A_{23}&\cdots&A_{2,n+1}\\
    \vdots& \vdots& \vdots&\dots&\vdots\\
    0&0&0&\cdots&A_{n+1,n+1}\\
    \end{pmatrix},
\end{equation*}
where $A_{ij}$ is a $d_i\times d_j$ matrix. Cleary, the unipotent radical $U$ of $P$ is
\begin{equation*}
    \begin{pmatrix}
    I_{d_1}&A_{12}&A_{13}&\cdots&A_{1,n+1}\\
    0& I_{d_2}&A_{23}&\cdots&A_{2,n+1}\\
    \vdots& \vdots& \vdots&\dots&\vdots\\
    0&0&0&\cdots& I_{d_{n+1}}\\
    \end{pmatrix},
\end{equation*}
of which the nilpotency class is smaller than $n+1$.
\end{proof}

Motivated by Proposition \ref{equivalent}, we introduce the definition of nilpotent $G$-Higgs bundles and nilpotent flat $G$-bundles.

\begin{defn}\label{defn_nil}
A $G$-Higgs bundle $(E,\theta)$ (resp. flat $G$-bundle $(V,\nabla)$) is said to be \emph{nilpotent of exponent $\leq n$} if there exists a covering of $X$ by open affine subsets $U_i$ such that for each $U_i$, the set
$$\{\theta|_{U_i} (\partial)\ |\ \partial\ \text{is a section of }T_{U_i} \} \quad (resp. \text{ } \{\psi|_{U_i} (\partial)\ |\ \partial\ \text{is a section of }F^*_{U'_i/k}T_{U'_i} \})$$ lies in $\mathfrak{u}(\mathcal{O}_{U_i})$, where $\mathfrak{u}$ is the Lie algebra of a unipotent radical of a parabolic subgroup $P$ of $G$ such that the nilpotency class of the unipotent radical is smaller than $n+1$.
\end{defn}

\section{Nilpotent $G$-Higgs Bundles and Nilpotent Flat $G$-Bundles}\label{sect_HIG_MIC_G}

and suppose that the characteristic $p$ is good.

Let $G$ be a connected split reductive linear algebraic group. We introduce the following categories:
\begin{itemize}
\item ${\rm HIG}_0(X/k,G)$: the category of $G$-bundles on $X$;

\item ${\rm MIC}_0(X/k,G)$: the category of flat $G$-bundles on $X$ with trivial $p$-curvature;

\item ${\rm HIG}_{p-1}(X/k,G)$: the category of nilpotent Higgs bundles on $X$ of exponent $\leq p-1$;

\item ${\rm MIC}_{p-1}(X/k,G)$: the category of nilpotent flat $G$-bundles on $X$ of exponent $\leq p-1$.
\end{itemize}
Let $\mathcal{C}$ be a category given above. We define $\mathcal{C}^{s(s)}$ to be the full subcategory of $\mathcal{C}$ of $R$-(semi)stable objects. For instance, ${\rm HIG}^{s(s)}_{p-1}(X/k,G)$ is the category of $R$-(semi)stable nilpotent $G$-Higgs bundles on $X$ of exponent $\leq p-1$.

The goal of this section is proving an equivalence of categories:
\begin{thm}\label{thm_HIG_MIC_G}
Suppose that $X$ is $W_2(k)$-liftable and the characteristic $p$ is good. The categories ${\rm HIG}_{p-1}(X'/k,G)$ and ${\rm MIC}_{p-1}(X/k,G)$ are equivalent.
\end{thm}

The $W_2(k)$-liftable condition and the goodness of characteristics will be explained later. To prove the theorem, we construct inverse Cartier transform (as a functor) in \S\ref{subsect_inv_car_G}
\begin{align*}
    C^{-1}_{\rm exp}: {\rm HIG}_{p-1}(X'/k,G) \rightarrow {\rm MIC}_{p-1}(X/k,G),
\end{align*}
and Cartier transform (as a functor) in \S\ref{subsect_car_G}
\begin{align*}
    C_{\rm exp}: {\rm MIC}_{p-1}(X/k,G) \rightarrow {\rm HIG}_{p-1}(X'/k,G)
\end{align*}
via an approach called \emph{exponential twisting} introduced in \cite{LSZ15}. These two functors are inverse to each other, and then induces an equivalence of categories. Furthermore, the equivalence holds under stability conditions (Theorem \ref{thm_HIG_MIC_G_stab}).

Before we move to the construction of inverse Cartier and Cartier transform, we first review three classical results: \emph{classical Cartier descent}, \emph{Deligne-Illusie's lemma} and \emph{exponentiation}, which will be used in the proof.

\hspace{\fill}

\subsubsection*{\textbf{Classical Cartier descent}}
The classical Cartier descent is given as follows:	
\begin{thm}[Theorem 5.1 in \cite{Kat70}]\label{thm_Cartier_des_Katz}
There is an equivalence of categories between ${\rm HIG}_0(X'/k)$ and ${\rm MIC}_0(X/k)$.
\end{thm}

\noindent Given $(E,0) \in {\rm HIG}_0(X'/k)$, Katz constructed a canonical connection $\nabla_{can}$ on $F^*_{X/k}E$ and prove that the induced functor ${\rm HIG}_0(X'/k) \rightarrow {\rm MIC}_0(X/k)$ gives an equivalence of categories. Chen--Zhu extended this result to principal bundles:
\begin{cor}[Example A.5 in \cite{CZ15}]\label{cor_cartier_des_G}
Let $\mathcal{G}'$ be a smooth affine group scheme over $X'$ and define $\mathcal{G}:=F^*_1 \mathcal{G}'$. Then we have an equivalence between the category of $\mathcal{G}'$-torsors on $X'$ and the category of $\mathcal{G}$-local systems with trivial $p$-curvature. As a special case, when $\mathcal{G}' = X' \times G$, the categories ${\rm HIG}_0(X'/k, G)$ and ${\rm MIC}_0(X/k,G)$ are equivalent.
\end{cor}

\hspace{\fill}

\subsubsection*{\textbf{A lemma by Deligne-Illusie}}
Let $W_2(k)$ be the length two Witt vector ring. We fix a $W_2(k)$-lifting $X_2$ of $X$ and choose an affine open cover $\{U_{2,i}\}_{i \in I}$ of $X_2$, where $I$ is the index set. Let $F_{2,i}: U_{2,i} \rightarrow U'_{2,i}$ be a Frobenius lifting, which ${\rm mod}\, p$ is the relative Frobenius morphism $F_{1,i} : U_{i} \rightarrow U'_{i}$, and we use the notation $F_1:=F_{1,i}$ for simplicity. We have a natural morphism
\begin{align*}
    \zeta_i := \frac{F^*_{2,i}}{[p]}: F^*_{1} \Omega_{U'_{i}} \rightarrow \Omega_{U_{i}}.
\end{align*}
Given $i,j,k \in I$, define $U_{ij}:=U_{i} \cap U_{j}$ and $U_{ijk}:=U_{i} \cap U_{j} \cap U_{k}$ to be the intersections.

\begin{lem}[\S 2 in \cite{DI87} or Lemma 2.1 in \cite{LSZ15}]\label{lem_DI}
There exist homomorphisms $h_{ij}: F_1^* \Omega_{U'_{ij}} \rightarrow \mathcal{O}_{U_{ij}}$ such that
\begin{enumerate}
\item $\zeta_{i}-\zeta_{j} = d h_{ij}$ over $U_{ij}$,
\item $h_{ij} + h_{jk} = h_{ik}$ over $U_{ijk}$.
\end{enumerate}
\end{lem}

\hspace{\fill}

\subsubsection*{\textbf{Exponentiation}}

Sobaje gave the definition of exponential maps for commuting nilpotent varieties \cite[Definition 2.3]{Sob15b} and proved the uniqueness of the exponential maps when the characteristic $p$ is good. We refer the reader to \cite[\S 2.9]{Her13} for the definition of good characteristics.

\begin{thm}[Theorem 3.4 in \cite{Sob15b}]\label{thm_sob3.4}
Suppose that the characteristic $p$ is good. Let $\mathcal{N}_1(\mathfrak{g})$ be the $[p]$-nilpotent variety of $\mathfrak{g}$ and let $\mathcal{U}_1(G)$ be the $p$-unipotent variety of $G$. Then there is a unique $G$-equivariant bijection
\begin{align*}
    \overline{\rm exp} : \mathcal{N}_1(\mathfrak{g}) \xrightarrow{\cong} \mathcal{U}_1(G),
\end{align*}
which is an exponential map, satisfying the property that if $U$ is the unipotent radical of a parabolic subgroup $P \subseteq G$ such that $U$ has nilpotency class less than $p$, then $\overline{\rm exp}$ restricts to a $P$-equivariant isomorphism $\mathfrak{u} \xrightarrow{\cong} U$ of algebraic groups, where $\mathfrak{u}$ is the Lie algebra of $U$, having tangent maps equal to the identity on $\mathfrak{u}$.
\end{thm}

\begin{rem}\label{rem_mod_p_red}
When the nilpotency class of $U$ is less than $p$, the $P$-equivariant isomorphism $\mathfrak{u} \xrightarrow{\cong} U$ is first studied by Serre \cite[\S 4]{Ser94} and then proved by Seitz \cite[\S 5]{Sei00}, where Seitz proved the existence and uniqueness of the exponential map $\mathfrak{u} \rightarrow U$ via an argument of mod $p$ reduction. Since this approach is used in this section, we briefly review it as follows. We fix a maximal torus $T$ of $G$ which is contained in $P$, and let $\mathcal{R}^+$ be the set of positive roots.  Let $\Gamma \subset\mathcal{R}^+$ be a closed subset of positive roots and denote by $U_\Gamma$ the corresponding unipotent subgroup with Lie algebra $\mathfrak{u}_\Gamma$. Let $i : \Spec\, k\rightarrow \Spec \, \mathbb{Z}_{(p)}$ be the composite $\Spec \, k \rightarrow \Spec \, \mathbb{F}_p \rightarrow \Spec \, \mathbb{Z}_{(p)}$. Then the exponential isomorphism $\overline{\rm exp}:\mathfrak{u}_\Gamma \rightarrow U_\Gamma$ in characteristic $p$ comes from the base change of the exponential isomorphism $$\widetilde{\rm exp}:\widetilde{\mathfrak{u}}_\Gamma \rightarrow \widetilde{U}_\Gamma$$ in characteristic $0$, where $\widetilde{U}_\Gamma$ (resp. $\widetilde{\mathfrak{u}}_\Gamma$) is the unipotent subgroup (resp. nilpotent Lie algebra) over $\mathbb{Z}_{(p)}$ such that $U_\Gamma = \widetilde{U}_\Gamma \times_{\Spec \, \mathbb{Z}_{(p)}} \Spec \, k$. We refer the reader to \cite[\S 5]{Sei00} for more details. Moreover, we can identifies $U_\Gamma$ with the affine space with coordinates $(x_\alpha)_{\alpha\in \Gamma}$, then the exponential isomorphism can be expressed by the formula $$\overline{\rm exp}(D)=\bigg(\sum_{0\leq m\leq p-1}\dfrac{D^m}{m!}(x_\alpha)\ \text{at}\ e\bigg)_{\alpha\in \Gamma},$$ and we refer the reader to \cite[\S 6]{BDP17} for more details.
\end{rem}

\subsection{Inverse Cartier transform}\label{subsect_inv_car_G}
In this subsection, we will show that given a nilpotent $G$-Higgs bundle of exponent $\leq p-1$, it can be associated with a nilpotent flat $G$-bundle of exponent $\leq p-1$. Furthermore, this construction induces a functor
\begin{align*}
C^{-1}_{\rm exp}: {\rm HIG}_{p-1}(X'/k,G) \rightarrow {\rm MIC}_{p-1}(X/k,G).
\end{align*}
We fix some notations first. Let $\mathcal{U}=\{U_i\}_{i \in I}$ be a covering of open affine subsets of $X$ given in Lemma \ref{lem_DI}. Given a $G$-Higgs bundle $(E,\theta)$ on $X'$, denote by $(E_i,\theta_i)$ the restriction of $(E,\theta)$ to $U'_i$. We can define the above notations on the intersections $U_{ij}$ similarly. For convenience, we use the same notation $F_1$ for relative Frobenius morphisms of $X$, $U_i$ and etc. for simplicity.
	
Let $(E,\theta)$ be a $G$-Higgs bundle in ${\rm HIG}_{p-1}(X'/k,G)$. The element $h_{ij}(F_1^* \theta) := h_{ij}(F_1^* \theta_{ij})$ is a section of $(F_1^* E_{ij})(\mathfrak{g})$. Since $\theta$ is nilpotent of exponent $\leq p-1$, the set $\{\theta_{ij} (\partial)\ |\ \partial\ \text{is a section of }T_{U'_{ij}} \}$ lies in a Lie algebra $\mathfrak{u}(\mathcal{O}_{U'_{ij}})$, which is the Lie algebra of a unipotent subgroup $U$. Then by Theorem \ref{thm_sob3.4}, there exists an isomorphism $\overline{\rm exp} : \mathfrak{u} \xrightarrow{\cong} U$ given by the exponentiation. Therefore, the element
\begin{align*}
G_{ij} := \overline{\exp}(h_{ij}(F_1^* \theta_{ij}))
\end{align*}
is well-defined and induces an automorphism of $F^*_1 E_{ij}$. If there is no ambiguity, we use the notation
\begin{align*}
    h_{ij}(F_1^* \theta) := h_{ij}(F_1^* \theta_{ij})
\end{align*}
for simplicity. Since $\theta$ is integrable, we have
\begin{align*}
\overline{\exp} ( h_{ij} (F^*_1 \theta) ) \cdot \overline{\exp}  ( h_{jk} (F^*_1 \theta) ) = \overline{\exp} (( h_{ij} + h_{jk}) (F^*_1 \theta) ).
\end{align*}
By Lemma \ref{lem_DI}, the data $\{G_{ij}\}_{i,j \in I}$ satisfies the cocycle condition
\begin{equation}\label{eq_G}
G_{ij} \cdot G_{jk} = G_{ik}.
\end{equation}
Gluing $\{V_i:=F^*_1 E_i\}_{i \in I}$ via $G_{ij}$, we obtain a $G$-bundle $V$ on $X$.
	
Next, we construct a connection on $V$. Denote by $\nabla_{can,i}$ the canonical connection on $F^*_1 E_i$ given by Corollary \ref{cor_cartier_des_G}, and we define a new connection
\begin{equation}\label{eq_new_conn}
\nabla_i = \nabla_{can,i} + \zeta_i(F^*_1 \theta)
\end{equation}
on $F^*_1 E_i$. We have to show that the local data $\{\nabla_i\}_{i \in I}$ can be glued together via $\{G_{ij}\}_{i,j \in I}$, i.e.
\begin{equation}\label{eq_nabla}
G_{ij} \circ \nabla_j = \nabla_i,
\end{equation}
where the action is the \emph{gauge action}. It is equivalent to show
\begin{align*}
G_{ij}^{-1} \cdot dG_{ij} +  {\rm Ad}(G_{ij}) ( \zeta_j(F^*_1 \theta)) =  \zeta_i (F^*_1 \theta).
\end{align*}
We claim
\begin{align*}
    G^{-1}_{ij} \cdot dG_{ij} = dh_{ij} (F^*_1 \theta), \quad {\rm Ad}(G_{ij}) ( \zeta_j(F^*_1 \theta)) = \zeta_j(F^*_1 \theta),
\end{align*}
and then, the left-hand side of equation is
\begin{equation}\tag{$\ast$}\label{eq_char_p}
    G_{ij}^{-1} \cdot dG_{ij} +  {\rm Ad}(G_{ij}) ( \zeta_j(F^*_1 \theta)) = dh_{ij} (F^*_1 \theta) + \zeta_j(F^*_1 \theta).
\end{equation}
Therefore, the desired equality follows directly from Lemma \ref{lem_DI} and we obtain a $G$-connection $\nabla$ on $V$. The claim can be proved by a mod $p$ reduction argument and we only give the proof of the equality $$G^{-1}_{ij} \cdot dG_{ij} = dh_{ij} (F^*_1 \theta)$$ as an example. Let $\rho:G\hookrightarrow {\rm GL}(V)$ be a faithful finite dimensional representation given by the regular representation (see \cite[\S 8.6]{Hum75} or \cite[Theorem 4.9]{Mil17} for instance). Let $m$ be an integer such that $p^m>n$ and we take a $W_m(k)$-lifting $\widetilde{U}_{ij}$ of $U_{ij}$. Note that such a lifting always exists because $U_{ij}$ is affine. Moreover, there also exists a $W_m(k)$-lifting $\widetilde{\rho}:\widetilde{G}\hookrightarrow {\rm GL}(\widetilde{V})$ by the construction of $\rho$. Then we have the following equality for the corresponding Lie algebras
\begin{align*}
    p \widetilde{\mathfrak{g}} = (p \mathfrak{gl}(\widetilde{V}) ) \cap \widetilde{\mathfrak{g}}
\end{align*}
because the induced morphism $d \widetilde{\rho}: \widetilde{\mathfrak{g}} \rightarrow \mathfrak{gl}(\widetilde{V})$ of Lie algebras is injective. Let $\theta_{ij}$ be the restriction of the Higgs field $\theta$ to $U_{ij}$. We choose a $W_m(k)$-lifting $\widetilde{\theta}_{ij}$ of $\theta_{ij}$. The equality \eqref{eq_char_p} in characteristic $p$ is equivalent to the following equality:
\begin{align*}
    \widetilde{G}_{ij}^{-1} \cdot d\widetilde{G}_{ij} \equiv d\widetilde{h}_{ij} (\widetilde{F}^*_1 \widetilde{\theta}_{ij}) \text{ mod } \left( p (\widetilde{\mathfrak{g}} \otimes_{ \Spec \, W_m(k)} \Omega_{\widetilde{U}_{ij}} )(\mathcal{O}_{\widetilde{U}_{ij}}) \right),
\end{align*}
where $\widetilde{h}_{ij}$ (resp. $\widetilde{F}^*_1)$ is the lifting of $h_{ij}$ (resp. $F^*_1$) and $\widetilde{G}_{ij}=\widetilde{\rm exp}(   \widetilde{h}_{ij} (\widetilde{F}^*_1 \widetilde{\theta}_{ij})  )$.  By the facts that the exponential map commutes with $d\widetilde{\rho}$ and the equality $p \widetilde{\mathfrak{g}} = (p \mathfrak{gl}(\widetilde{V}) ) \cap \widetilde{\mathfrak{g}}$, it is equivalent to show the following equality in the case of ${\rm GL}(V)$
\begin{align*}
    {\rm exp}(A)^{-1} \cdot d {\rm exp}(A) \equiv dA \text{ mod } \left( p (\mathfrak{gl}(\widetilde{V}) \otimes_{ \Spec \, W_m(k)} \Omega_{\widetilde{U}_{ij}} )(\mathcal{O}_{\widetilde{U}_{ij}}) \right),
\end{align*}
where $A=d\widetilde{\rho}(\widetilde{h}_{ij} (\widetilde{F}^*_1 \widetilde{\theta}_{ij})) \in \mathfrak{gl}(\widetilde{V})(\mathcal{O}_{ \widetilde{U}_{ij}})$ and ${\rm exp}(A)=\sum\limits_{0\leq m\leq n-1}\dfrac{A^m}{m!}$. Then we have
\begin{align*}
    d A^m \equiv mA^{m-1}dA \text{ mod } \left( p (\mathfrak{gl}(\widetilde{V}) \otimes_{ \Spec \, W_m(k)} \Omega_{\widetilde{U}_{ij}} )(\mathcal{O}_{\widetilde{U}_{ij}}) \right).
\end{align*}
Therefore,
 \begin{align*}
     {\rm exp}(A)^{-1}\cdot d{\rm exp}(A) & \equiv {\rm exp}(-A)(\sum\limits_{0\leq m\leq n-2}\dfrac{A^m}{m!}dA) \equiv {\rm exp}(-A) {\rm exp}(A)dA \\
     & \equiv dA \text{ mod } \left( p (\mathfrak{gl}(\widetilde{V}) \otimes_{ \Spec \, W_m(k)} \Omega_{\widetilde{U}_{ij}} )(\mathcal{O}_{\widetilde{U}_{ij}}) \right).
\end{align*}
This finishes the proof of the equality \eqref{eq_char_p}. We  remark here that we can not show this equality directly by choose a faithful representation of $G$ in characteristic $p$ since the exponential map may not commute with morphisms of algebraic groups in characteristic $p$.
	
Finally, we will prove that the $G$-connection $\nabla$ is integrable and nilpotent of exponent $\leq p-1$. Since either integrability and nilpotency is a local property, we only have to check them on local charts $U_i$. Then the proof of the integrability is exactly the same as \cite[\S 2.2 Step 3]{LSZ15}. For the nilpotency, the nilpotency of the $p$-curvature of $\nabla_i$ is equivalent to that of $F^*_1 \theta_i$, and thus, $\nabla_i$ is also nilpotent of exponent $\leq p-1$.

In conclusion, the above construction gives a functor
\begin{align*}
C^{-1}_{\rm exp}: {\rm HIG}_{p-1}(X'/k,G) \rightarrow {\rm MIC}_{p-1}(X/k,G).
\end{align*}
	
\subsection{Cartier transform}\label{subsect_car_G}
We consider the other direction: given a flat $G$-bundle $(V,\nabla) \in {\rm MIC}_{p-1}(X/k,G)$, we want to construct a $G$-Higgs bundle $(E,\theta) \in {\rm HIG}_{p-1}(X'/k,G)$, which will induce a functor
\begin{align*}
C_{\rm exp}: {\rm MIC}_{p-1}(X/k,G) \rightarrow {\rm HIG}_{p-1}(X'/k,G).
\end{align*}

Given a pair $(V,\nabla) \in {\rm MIC}_{p-1}(X/k,G)$, denote by $\psi$ the $p$-curvature of $\nabla$. Note that $\psi$ is nilpotent of exponent $\leq p-1$ by definition. We use the same notation for coverings as in \S\ref{subsect_inv_car_G}, and let $V_i:= V|_{U_i}$, $\psi_i:= \psi|_{U_i}$. Then, we define a new connection $\nabla'_i$ on $V_i$ as
\begin{align*}
\nabla'_i:=\nabla_i + \zeta_i(\psi).
\end{align*}
We regard the new connection as $\nabla'_i: \mathcal{O}_{V_i} \rightarrow \mathcal{O}_{V_i} \otimes \Omega_{U_i}$ (see Remark \ref{rem_Higgs_conn}), and then the $p$-curvature of $\nabla'_i$ is trivial by following the same argument as in \cite[\S 2.3, Step 3, page 866-869]{LSZ15}. For completeness of the proof, we also include it here. Consider the algebra $$F_1^*(S^{<p}T_{U'_i}):=\bigoplus_{0\leq m<p} F_1^*(S^mT_{U'_i})=F_1^*(S\dot\ T_{U'_i})/I,$$  where $I$ is the ideal $F_1^*(S^{\geq p}T_{U'_i}):=\bigoplus\limits_{ m\geq p} F_1^*(S^mT_{U'_i})$. There exists a connection $\nabla_T$ on $F_1^*(S^{< p}T_{U'_i})$ with $p$-curvature $\psi_{\nabla_T}(\partial)(\tau)=\partial\tau$ for $\partial\in F_1^*T_{U'_i},\ \tau\in F_1^*(S^{<p}T_{U'_i})$ by \cite[\S 2.3,  page 868]{LSZ15}. We have an isomorphism of $\mathcal{O}_{U_i}$-modules:
$$\lambda:\mathscr{H}om_{F_1^*(S^{< p}T_{U'_i})}(F_1^*(S^{< p}T_{U'_i}),\mathcal{O}_{V_i})\xrightarrow{\cong} \mathcal{O}_{V_i},\quad \phi\mapsto \phi(1).$$ Let  $\Tilde{\nabla}$ be the connection on $\mathscr{H}om_{F_1^*(S^{< p}T_{U'_i})}(F_1^*(S^{< p}T_{U'_i}),\mathcal{O}_{V_i})$ given by the formula $$(\Tilde{\nabla}(\partial)(\phi))(\tau)+\phi(\nabla_T(\partial)(\tau))=\nabla_i(\partial)(\phi(\tau))$$ for $\partial\in F_1^*T_{U'_i},\phi\in \mathscr{H}om_{F_1^*(S^{< p}T_{U'_i})}(F_1^*(S^{< p}T_{U'_i}),\mathcal{O}_{V_i}), \tau\in F_1^*(S^{< p}T_{U'_i}),$ then the $p$-curvature $\psi_{\Tilde{\nabla}}$ of $\Tilde{\nabla}$ is tautologically zero since
$$(\psi_{\Tilde{\nabla}}(\partial)(\phi))(\tau)=\psi_{\nabla_i}(\partial)(\phi(\tau))-\phi(\psi_{\nabla_T}(\partial)(\tau))=0$$ hold for all $\partial\in F_1^*T_{U'_i}, \phi\in \mathscr{H}om_{F_1^*(S^{< p}T_{U'_i})}(F_1^*(S^{< p}T_{U'_i}),\mathcal{O}_{V_i}), \tau\in F_1^*(S^{< p}T_{U'_i}). $ This implies that $\nabla_i^\prime$ has trivial $p$-curvature since $\Tilde{\nabla}=\nabla_i^\prime$ via the above isomorphism $\lambda$.

Since the $p$-curvature $\psi_i$ is horizontal with respect to $\nabla_i$, it is also horizontal with respect to $\nabla'_i$. Now we want to glue $\{(V_i,\nabla'_i,\psi_i)\}_{i \in I}$ to obtain a new flat $G$-bundle $(V',\nabla')$, of which the $p$-curvature is trivial, together with a $F$-Higgs field $\psi$ horizontal with respect to $\nabla'$. We define
\begin{align*}
J_{ij}:= \overline{\exp} (h_{ij}(\psi)),
\end{align*}
which is an automorphism of $V_{ij}$. With the same argument as \eqref{eq_G}, we have
\begin{align*}
    J_{ij} \cdot J_{jk} = J_{ik}, \quad J_{ij} \circ \psi_j = \psi_i,
\end{align*}
where $J_{ij} \circ \psi_j$ is the adjoint action, and then we obtain a new $G$-bundle $V'$ and a new section $\psi'$ of $V'(\mathfrak{g}) \otimes F^*_1 \Omega_{X'}$, which is regarded as a $F$-Higgs field. A similar calculation of formula \eqref{eq_nabla} gives
\begin{align*}
    \quad J_{ij} \circ \nabla'_j = \nabla'_i,
\end{align*}
and we get an integrable $G$-connection $\nabla'$ on $V'$. In conclusion, we get a triple $(V',\nabla',\psi')$, where $(V',\nabla')$ is a flat $G$-bundle with trivial $p$-curvature and $\psi'$ is a $F$-Higgs field on $V'$. Since $\psi'$ is horizontal with respect to $\nabla'$, of which the $p$-curvature is zero, $(V',\nabla',\psi')$ descends to a $G$-Higgs bundle $(E,\theta)$ by Corollary \ref{cor_cartier_des_G}. Therefore, we obtain a functor
\begin{align*}
C_{\rm exp}: {\rm MIC}_{p-1}(X/k,G) \rightarrow {\rm HIG}_{p-1}(X'/k,G).
\end{align*}
as expected.

\subsection{Stability condition}\label{subsect_stab_cond}

Suppose that $X$ is smooth and projective. We fix an ample line bundle $H$ on $X$. Define $H':=\pi_{X/k}^* H$, which is also an ample line bundle on $X'$. Furthermore, $H^p \cong F_{X/k}^* H'\cong F_X^* H$. We remind the reader that the stability condition of $G$-Higgs bundles on $X'$ is given by the ample line bundle $H'$, while the stability condition of flat $G$-bundles on $X$ is given by $H$.

We first improve the equivalence given in Corollary \ref{cor_cartier_des_G} by adding the information of stability conditions:
\begin{prop}\label{prop_cart_des_G_stab}
The category ${\rm HIG}^{(s)s}_0(X'/k,G)$ is equivalent to ${\rm MIC}^{(s)s}_0(X/k,G)$.
\end{prop}

\begin{proof}
Let $E$ be a $G$-bundle on $X'$, and denote by $(V,\nabla_{can})$ the corresponding flat $G$-bundle on $X$ with trivial $p$-curvature by Corollary \ref{cor_cartier_des_G}, where $V = F^*_{1} E$. Let $P$ be a parabolic subgroup of $G$. To prove the equivalence of stability conditions, it is enough to prove the following two statements:
\begin{enumerate}
\item there is a one-to-one correspondence between reductions of structure group $\sigma: (X^{\rm big})' \rightarrow (E|_{(X^{\rm big})'}) /P$ and $\nabla_{can}$-compatible reductions of structure group $\varsigma: X^{\rm big} \rightarrow (V|_{X^{\rm big}})/P$;
\item we have
\begin{align*}
\deg_H L(\varsigma,\chi) = p \deg_{H'} L(\sigma,\chi),
\end{align*}
where reductions $\sigma$ and $\varsigma$ are given by the first statement and $\chi: P \rightarrow \mathbb{G}_m$ is a character.
\end{enumerate}
For the first statement, although the stability condition (Definition \ref{defn_R_stab_G} and \ref{defn_stab_lambda}) is given for reductions of structure group over $X^{\rm big}$, it is enough to work on $X$. given a reduction of structure group $\sigma: X' \rightarrow E/P$, we obtain a reduction $\varsigma: X \rightarrow V/P$ via the relative Frobenius morphism $F_{1}$. Equivalently, the morphism $E_\sigma \rightarrow E$ induces a morphism $(V_{\varsigma},\nabla_{can,\varsigma}) \rightarrow (V,\nabla_{can})$ of flat $G$-bundles by the classical Cartier transform (Theorem \ref{thm_Cartier_des_Katz}), where $V_{\varsigma} = F^*_{1} E_{\sigma}$. Therefore, $\varsigma$ is $\nabla_{can}$-compatible. By Theorem \ref{thm_Cartier_des_Katz}, the correspondence between reduction of structure groups $\sigma$ and $\varsigma$ is a one-to-one correspondence. The second statement is clear with respect to the correspondence given in the first statement.
\end{proof}

Now we consider the case of nilpotent $G$-Higgs bundles and nilpotent flat $G$-bundles of exponent $\leq p-1$.
\begin{thm}\label{thm_HIG_MIC_G_stab}
The categories ${\rm HIG}^{(s)s}_{p-1}(X'/k,G)$ and ${\rm MIC}^{(s)s}_{p-1}(X/k,G)$ are equivalent.
\end{thm}

\begin{proof}
Let $(E,\theta) \in {\rm HIG}_{p-1}(X'/k,G)$ be a $G$-Higgs bundle, and denote by $(V,\nabla) \in {\rm MIC}_{p-1}(K/k,G)$ the corresponding flat $G$-bundle under the equivalence given by Theorem \ref{thm_HIG_MIC_G}. With the same approach as Proposition \ref{prop_cart_des_G_stab}, it is enough to show that $\theta$-compatible reductions $\sigma: X' \rightarrow E/P$ are in one-to-one correspondence with $\nabla$-compatible reductions $\varsigma: X \rightarrow V/P$.

\hspace{\fill}

Let $\sigma: X' \rightarrow E/P$ be a $\theta$-compatible reduction. Denote by $\theta_\sigma: X' \rightarrow E_\sigma(\mathfrak{p}) \otimes \Omega_{X'}$ the lifting of $\theta$ via $\sigma$. We work on local charts $U_i$ first and construct a $\nabla_i$-compatible reduction $\varsigma_i$, where $\nabla_i:=\nabla|_{U_i}$, and then we will show that the local data $\varsigma_i$ can be glued together. We use the same notation as we did in \S\ref{subsect_inv_car_G} and \S\ref{subsect_car_G}. Denote by $\sigma_i=\sigma|_{U_i}$ the restriction to $U_i$. In \S\ref{subsect_inv_car_G}, we construct $(V_i,\nabla_i)$ on $U_i$, where $V_i = F^*_1 E_i$ and $\nabla_i = \nabla_{can} + \zeta_i(F^*_1 \theta)$. Let $\varsigma_i: U_i \rightarrow V_i/P$ be the corresponding reduction of $\sigma_i$ given in Proposition \ref{prop_cart_des_G_stab}, and then $\varsigma_i$ is compatible with $\nabla_{can}$. Since $\sigma_i$ is compatible with $\theta_i$, the reduction $\varsigma_i$ is compatible with $F^*_1 \theta_i$ and thus compatible with $\zeta_i(F^*_1 \theta)$. In conclusion, the reduction $\varsigma_i$ is compatible with $\nabla_i$, and we denote by $\nabla_{\varsigma,i}$ the induced $P$-connection on $V_{\varsigma,i}$. Then, we have the following diagrams on $U_i$.
\begin{center}
\begin{tikzcd}
V_{\varsigma,i} \arrow[r] \arrow[d] & V_i \arrow[d] \\
U_i \arrow[r,"\varsigma_i"] & V_i/P
\end{tikzcd}\quad \quad \quad
\begin{tikzcd}
& V_{\varsigma,i}(\mathfrak{p}) \arrow[d] \\
T_{U_i} \arrow[r,"\nabla_i"] \arrow[ur,"\nabla_{\varsigma,i}"] & V_i(\mathfrak{g}) .
\end{tikzcd}
\end{center}
Moreover, we obtain a flat $P$-bundle $(V_{\varsigma},\nabla_{\varsigma})$ by gluing local data $(V_{\varsigma,i},\nabla_{\varsigma,i})$ via
\begin{align*}
    G_{\sigma,ij}:= \overline{\exp}(h_{ij}(F^*_1( \theta_\sigma )) )
\end{align*}
with the same argument as in \S\ref{subsect_inv_car_G}.

Now we will show that $\varsigma_i$ can be glued together via $G_{ij}$, and it is equivalent to prove that the morphisms $\{V_{\varsigma,i} \rightarrow V_i\}_{i \in I}$ can be glued together. For convenience, we use the same notation $\varsigma_i$ for the induced morphism $V_{\varsigma,i} \rightarrow V_i$, and we have to show that the following diagram commutes.
\begin{center}
\begin{tikzcd}
V_{\varsigma,j}(\mathfrak{p}) \arrow[r, "G_{\sigma,ij}"] \arrow[d,"\varsigma_j"] & V_{\varsigma,i}(\mathfrak{p}) \arrow[d,"\varsigma_i"] \\
V_j(\mathfrak{g}) \arrow[r,"G_{ij}"] & V_i(\mathfrak{g})
\end{tikzcd}
\end{center}
By the compatibility of $\nabla_i$ and $\nabla_j$, we have
\begin{align*}
    \nabla_j = \varsigma_j \circ \nabla_{\varsigma,j}, \quad \nabla_i = \varsigma_i \circ \nabla_{\varsigma,i}.
\end{align*}
Since $\{\nabla_{\varsigma,i}, \, i \in I\}$ and $\{ \nabla_{i}, \,  i \in I\}$ can be glued together via $G_{\sigma,ij}$ and $G_{ij}$ respectively, we have
\begin{align*}
    G_{\sigma,ij} \circ \nabla_{\varsigma,j} = \nabla_{\varsigma,i}, \quad G_{ij} \circ \nabla_{j} = \nabla_{i}.
\end{align*}
Thus, we obtain the desired property
\begin{align*}
    G_{ij} \circ \varsigma_j = \varsigma_i \circ G_{\sigma,ij}
\end{align*}
by the following diagram.
\begin{center}
\begin{tikzcd}
& T_{U_{ij}} \arrow[ld, "\nabla_{\varsigma,j}" description] \arrow[lddd, "\nabla_j" description] \arrow[rd, "\nabla_{\varsigma,i}" description] \arrow[rddd, "\nabla_i" description] & \\
V_{\varsigma,j}(\mathfrak{p}) \arrow[rr, "G_{\sigma,ij}"] \arrow[dd,"\varsigma_j"] & & V_{\varsigma,i}(\mathfrak{p}) \arrow[dd,"\varsigma_i"] \\
& & & \\
V_j(\mathfrak{g}) \arrow[rr,"G_{ij}"] & & V_i(\mathfrak{g})
\end{tikzcd}
\end{center}
In conclusion, we obtain a $\nabla$-compatible reduction $\varsigma: X \rightarrow V/P$.

\hspace{\fill}

The other direction is proved in a similar way. Let $\varsigma:X \rightarrow V/P$ be a $\nabla$-compatible reduction. Clearly, $\varsigma$ is also compatible with the $p$-curvature $\psi$. Denote by $\nabla_\varsigma$ (resp. $\psi_{\varsigma}$) the induced connection ($p$-curvature) by $\nabla$. Recall that we defined a triple $(V_i,\nabla'_i,\psi_i)$ on each $U_i$, where
\begin{align*}
    V_i=V|_{U_i}, \quad \nabla'_i = \nabla_i + \zeta_i(\psi), \quad \psi_i=\psi|_{U_i}.
\end{align*}
Since $\varsigma$ is compatible with $\nabla$ and $\psi$, the restriction $\varsigma_i:=\varsigma|_{U_i}$ is also compatible with $\nabla_i$ and $\psi_i$ (thus compatible with $\nabla'_i$). Let $V_{\varsigma,i}$ be the restriction of the $P$-bundle $V_{\varsigma,i}$ to $U_i$, and denote by $\nabla_{\varsigma,i}$ the induced connection on $V_{\varsigma,i}$. Abusing the notation, we use the same notation $\varsigma_i: V_{\varsigma,i} \rightarrow V_{i}$ for the induced morphism. Locally, we have the following diagram
\begin{center}
\begin{tikzcd}
& T_{U_{ij}} \arrow[ld, "\psi_{\varsigma,j}" description] \arrow[lddd, "\psi_j" description] \arrow[rd, "\psi_{\varsigma,i}" description] \arrow[rddd, "\psi_i" description] & \\
V_{\varsigma,j}(\mathfrak{p}) \arrow[rr, "J_{\varsigma,ij}"] \arrow[dd,"\varsigma_j"] & & V_{\varsigma,i}(\mathfrak{p}) \arrow[dd,"\varsigma_i"] \\
& & & \\
V_j(\mathfrak{g}) \arrow[rr,"J_{ij}"] & & V_i(\mathfrak{g})
\end{tikzcd}
\end{center}
where
\begin{align*}
    J_{\varsigma,ij}:=\overline{\exp}(h_{ij}(\psi_{\varsigma})),
\end{align*}
such that
\begin{itemize}
    \item $\psi$-compatibility:
    \begin{align*}
        \psi_j = \varsigma_j \circ \psi_{\varsigma,j}, \quad \psi_i = \varsigma_i \circ \psi_{\varsigma,i}.
    \end{align*}
    \item gluing property:
    \begin{align*}
        J_{\varsigma,ij} \circ \psi_{\varsigma,j} = \psi_{\varsigma,i}, \quad  J_{ij} \circ \psi_{j} = \psi_{i}.
    \end{align*}
\end{itemize}
Therefore,
\begin{align*}
    J_{ij} \circ \varsigma_j = \varsigma_i \circ J_{\varsigma,ij},
\end{align*}
and then the local data $\{\varsigma_i\}_{i \in I}$ can be glued via $J_{ij}$, and denote it by $\varsigma'$. Clearly, the reduction $\varsigma'$ is compatible with both $\nabla'$ and $\psi'$. By Proposition \ref{prop_cart_des_G_stab}, $(V',\nabla')$ descent to a $G$-bundle $E$ on $X'$. Thus, $\varsigma'$ descents to a reduction $\sigma: X' \rightarrow E/P$. Since $\varsigma'$ is $\psi'$-compatible, the reduction $\sigma$ is $\theta$-compatible. This finishes the proof.
\end{proof}

\section{Logarithmic G-Higgs bundles and Flat G-bundles on Root Stacks}

In this section, we are mostly interested in a special classes of smooth Deligne--Mumford stacks $\mathscr{X}$ considered and discussed in \cite{MO05,Cad07,Sim11}, which are called \emph{root stacks} in this paper (Theorem \ref{thm_root_stack}). We prove an equivalence of categories between logarithmic nilpotent $G$-Higgs bundles and logarithmic nilpotent flat $G$-bundles of exponent $\leq p-1$ on root stacks (Theorem \ref{thm_HIG_MIC_G_log_stack}), which also preserves stability conditions (Theorem \ref{thm_HIG_MIC_G_stab_stack}).

\subsection{G-Higgs Bundles and Flat G-bundles on DM stacks}\label{subsect_Higgs_conn_stack}

Recall that a \emph{Deligne--Mumford stack} over $k$ is an algebraic stack such that there exists a surjective \'etale morphism $Y \rightarrow \mathscr{X}$, where $Y$ is a scheme. It is well-known that $\mathscr{X} \cong [Y/R]$, where $R= Y \times_{\mathscr{X}} Y$ is an algebraic space together with two \'etale morphisms
\begin{align*}
    p_i: R = Y \times_{\mathscr{X}} Y \rightarrow Y
\end{align*}
for $i=1,2$. For convenience, we suppose that $R$ is a scheme. In fact, we can apply all arguments in this section to algebraic spaces first, and then move to Deligne--Mumford stacks. Let $F_1:=F_{\mathscr{X}/k}: \mathscr{X} \rightarrow \mathscr{X}'$ be the relative Frobenius morphism and we suppose that $\mathscr{X}$ is $W_2(k)$-liftable. Here is an equivalent description under the isomorphism $\mathscr{X} \cong [Y/R]$. The relative Frobenius morphism $F_{\mathscr{X}/k}$ is induced from those on $Y$ and $R$, and the condition that $\mathscr{X}$ is $W_2(k)$-liftable is equivalent to say that $Y$ and $R$ are $W_2(k)$-liftable and the following diagram commutes
\begin{center}
\begin{tikzcd}
    R_2 \arrow[r,"\widetilde{p}_i"] & Y_2 \arrow[r] & \Spec \, W_2(k) \\
    R \arrow[u] \arrow[r, "p_i"] & Y \arrow[r] \arrow[u] & \Spec \, k \arrow[u] \, ,
\end{tikzcd}
\end{center}
where $\widetilde{p}_i$ is an \'etale morphism for $i=1,2$. The stack $\mathscr{X}$ is called \emph{tame} if the induced functor
\begin{align*}
    \pi_*: {\rm QCoh}(\mathscr{X}) \rightarrow {\rm QCoh}(X)
\end{align*}
is exact, where ${\rm QCoh}$ is the category of quasi-coherent sheaves. We want to remind the reader that the functor $\pi_*$ is always exact in the case of characteristic zero. In positive characteristic, there are several equivalent conditions of tameness \cite[Theorem 3.2]{AOV08}.

Suppose that $\mathscr{X}$ is tame, smooth and $W_2(k)$-liftable with coarse moduli space $\pi: \mathscr{X} \rightarrow X$, where $X$ is a smooth variety over $k$. Denote by $\Omega_{\mathscr{X}} := \Omega_{\mathscr{X}/k}$ the relative cotangent sheaf. Given $\mathscr{X} \cong [Y/R]$, a \emph{$G$-bundle} $\mathscr{E}$ on $\mathscr{X}$ can be regarded as a pair $(E,\tau)$, where $E$ is a $G$-bundle on $Y$ and $\tau: p_1^* E \xrightarrow{\cong} p_2^*E$ is an isomorphism satisfying the cocycle condition on $Y \times_{\mathscr{X}} Y \times_{\mathscr{X}} Y$. In the same way, a \emph{$\lambda$-connection} on $\mathscr{X}$ is regarded as a triple $(E,\nabla,\tau)$ on $Y$, where $(E,\nabla)$ is a $\lambda$-connection on $Y$ and $\tau: p_1^* (E,\nabla) \rightarrow p_2^* (E,\nabla)$ is an isomorphism satisfying the cocycle condition. The nilpotency for $G$-Higgs bundles and flat $G$-bundles on $\mathscr{X} \cong [Y/ R]$ is defined as the nilpotency of the corresponding $G$-Higgs bundles and flat $G$-bundles on $Y$. Note that this definition does not depend on the choice of the surjective \'etale morphism $Y \rightarrow \mathscr{X}$. We define the following categories:
\begin{itemize}
    \item ${\rm HIG}_{0}(\mathscr{X}/k,G)$ is the category of $G$-bundles on $\mathscr{X}$,
    \item ${\rm MIC}_{0}(\mathscr{X}/k,G)$ is the category of flat $G$-bundles on $\mathscr{X}$ with trivial $p$-curvature,
    \item ${\rm HIG}_{p-1}(\mathscr{X}/k,G)$ is the category of nilpotent $G$-Higgs bundles on $\mathscr{X}$ of exponent $\leq p-1$,
    \item ${\rm MIC}_{p-1}(\mathscr{X}/k,G)$ is the category of nilpotent flat $G$-bundles on $\mathscr{X}$ of exponent $\leq p-1$.
\end{itemize}

\begin{prop}\label{prop_cartier_des_G_stack}
We have an equivalence of categories between ${\rm HIG}_0(\mathscr{X}'/k,G)$ and ${\rm MIC}_0(\mathscr{X}/k,G)$.
\end{prop}

\begin{proof}
With the same setup as above, we prove the equivalence via the isomorphism $\mathscr{X} \cong [Y/R]$. We have the following commutative diagram,
\begin{center}
\begin{tikzcd}
    R' \arrow[r, "p'_i"] & Y' \\
    R \arrow[u, "F_{R,1}"] \arrow[r, "p_i"] & Y \arrow[u, "F_{Y,1}"]
\end{tikzcd}
\end{center}
where $F_{R,1}$ and $F_{Y,1}$ are the relative Frobenius morphisms. Given a $G$-bundle $\mathscr{E}$ on $\mathscr{X}'$, it can be regarded as a pair $(E,\tau')$, where $E$ is a $G$-bundle on $Y'$ and $\tau': (p'_1)^* E \cong (p'_2)^* E$ is an isomorphism of $G$-bundles on $R'$ satisfying the cocycle condition. Since $F_{Y,1} \circ p = p' \circ F_{R,1}$, where $p=p_1, p_2$, we have $F_{R,1}^* p'^* E \cong p^* F_{Y,1}^* E$. By Corollary \ref{cor_cartier_des_G}, the $G$-bundle $E$ on $Y'$ corresponds to a flat $G$-bundle $(V,\nabla_{can})$ on $Y$, where $V = F^*_{Y,1} E$. On the other hand,
the $G$-bundle $p'^* E$ on $R'$ corresponds to a flat $G$-bundle on $R$. This flat $G$-bundle is isomorphic to $p^*(V,\nabla_{can})$ by the isomorphism $F_{R,1}^* p'^* E \cong p^* F_{Y,1}^* E$ and the uniqueness of the canonical connection. Thus, we obtain an isomorphism $\tau: p_1^* (V, \nabla_{can}) \cong p_2^*(V, \nabla_{can})$ induced from $\tau'$. The data $(
(V,\nabla_{can}), \tau )$ gives a flat $G$-bundle $(\mathscr{V},\nabla_{can})$ with trivial $p$-curvature on $\mathscr{X}$. The other direction can be proved in the same way.
\end{proof}

\begin{thm}\label{thm_HIG_MIC_G_stack}
The categories ${\rm HIG}_{p-1}(\mathscr{X}'/k,G)$ and $ {\rm MIC}_{p-1}(\mathscr{X}/k,G)$ are equivalent.
\end{thm}

\begin{proof}
The proof of this theorem is a stacky version of that of Theorem \ref{thm_HIG_MIC_G}. We give a proof under the isomorphism $\mathscr{X} \cong [Y/ R]$. Before we move to construct inverse Cartier and Cartier transform, we first enrich Lemma \ref{lem_DI} under \'etale morphisms. We have the following diagram
\begin{center}
\begin{tikzcd}
& & R'_2 \arrow[rrr,"\widetilde{p}'" description] & & & Y'_2 \\
& & & & & \\
R' \arrow[rrr, "p'" description] \arrow[uurr] & & & Y' \arrow[uurr] & &\\
& & R_2 \arrow[rrr,"\widetilde{p}" description] \arrow[uuu, "F_{R,2}" description] & & & Y_2 \arrow[uuu, "F_{Y,2}" description] \\
& & & & & \\
R \arrow[rrr,"p" description] \arrow[uuu, "F_{R,1}" description] \arrow[uurr] & & & Y \arrow[uuu, "F_{Y,1}" description] \arrow[uurr] & &
\end{tikzcd}
\end{center}
where
\begin{itemize}
    \item $p$ and $\widetilde{p}$ are regarded as \'etale morphisms $p_i$ and $\widetilde{p}_i$ respectively, and $p'$ and $\widetilde{p}'$ are the induced ones;
    \item $F_{Y,1}$ and $F_{R,1}$ are relative Frobenius morphisms, and $F_{Y,2}$ and $F_{R,2}$ are Frobenius liftings, of which mod $p$ are $F_{Y,1}$ and $F_{R,1}$ respectively.
\end{itemize}
We fix an open cover $\{U_{Y,i}\}_{i \in I}$ of $Y$ and denote by $\{ U_{R,i} \}_{i \in I}$ the corresponding open cover of $R$ under the \'etale morphism $p$. By Lemma \ref{lem_DI}, we have
\begin{align*}
    & \zeta_{\bullet, i} := \frac{ F^*_{\bullet,2} }{[p]} : F^*_{\bullet,1} \Omega_{U'_{\bullet, i}} \rightarrow \Omega_{U_{\bullet,i}} \\
    & h_{\bullet,ij} : F^*_{\bullet,1} \Omega_{U'_{\bullet, ij}} \rightarrow \Omega_{U_{\bullet,ij}}
\end{align*}
such that
\begin{align*}
    & \zeta_{\bullet,i} - \zeta_{\bullet,i} = dh_{\bullet,ij} \\
    & h_{\bullet,ij} + h_{\bullet,jk} = h_{\bullet,ik},
\end{align*}
where $\bullet = Y \text{ or } R$. Since $F_{Y,2} \circ \widetilde{p} = \widetilde{p}' \circ F_{R,2}$ and $F_{Y,1} \circ p = p' \circ F_{R,1}$, the morphisms
\begin{align*}
    \frac{\widetilde{p}^* F^*_{Y,2}}{[p]}: p^* F^*_{Y,1} \Omega_{U'_{Y,i}} \rightarrow p^* \Omega_{U_{Y,i}} \cong \Omega_{U_{R,i}}
\end{align*}
and
\begin{align*}
    \frac{F^*_{R,2} \widetilde{p}'^* }{[p]}: F^*_{R,1} p'^* \Omega_{U'_{Y,i}} \rightarrow \Omega_{U_{R,i}}
\end{align*}
are the same, i.e.
\begin{equation}\label{eq_etale_mor}
    \frac{\widetilde{p}^* F^*_{Y,2}}{[p]} = \frac{F^*_{R,2} \widetilde{p}'^* }{[p]}.
\end{equation}

Now given a $G$-Higgs bundle on $\mathscr{X}'$, we regard it as triple a $(E,\theta,\tau')$, where $(E,\theta)$ is a $G$-Higgs bundle on $Y'$ and $\tau' : (p'_1)^*(E,\theta) \cong (p'_2)^*(E,\theta)$ is an isomorphism satisfying the cocycle condition. The inverse Cartier transform in \S\ref{subsect_inv_car_G} gives a flat $G$-bundle $(V,\nabla) := C^{-1}_{\rm exp}(E,\theta)$ on $Y$. We will construct an isomorphism $\tau: p_1^* (V,\nabla) \cong p_2^* (V,\nabla)$ satisfying the cocycle condition and thus obtain a flat $G$-bundle on $\mathscr{X}$. Since $p'_i: R' \rightarrow Y'$ is \'etale, $(p'_i)^* (E,\theta)$ is a $G$-Higgs bundle on $R'$. Applying the inverse Cartier transform $C^{-1}_{\rm exp}$ on $R'$ again, we obtain flat $G$-bundles $C^{-1}_{\rm exp} ( (p'_i)^* (E,\theta) )$ on $R$ for $i=1,2$ together with an isomorphism
\begin{align*}
    C^{-1}_{\rm exp}( \tau' ) : C^{-1}_{\rm exp} ( (p'_1)^* (E,\theta) ) \cong C^{-1}_{\rm exp} ( (p'_2)^* (E,\theta) ).
\end{align*}
We claim that
\begin{align*}
    C^{-1}_{\rm exp} ( (p')^* (E,\theta) ) \cong p^* (V,\nabla) = p^* ( C^{-1}_{\rm exp}(E,\theta) ),
\end{align*}
where $p=p_1,p_2$, and thus we obtain the desired isomorphism $\tau:=C^{-1}_{\rm exp}( \tau' )$ satisfying the cocycle condition.

Although the claim has been proved in \cite{She21}, we include the proof here because the construction of the Cartier transform is exactly given in the same way. Let $(E_i,\theta_i)$ be the restriction of $(E,\theta)$ on $U'_{Y,i}$. The flat $G$-bundle $p^* ( C^{-1}_{\rm exp}(E,\theta) )$ is obtained by gluing the local data
\begin{align*}
    (p^* F^*_{Y,1} E_i, \ \nabla_{can,i} + p^* \zeta_{Y,i}(F^*_{Y,1} \theta_i))_{i \in I}
\end{align*}
via
\begin{align*}
    \overline{\rm exp}( p^* h_{Y,ij} (F^*_{Y,1} \theta_{ij}) ).
\end{align*}
Similarly, $C^{-1}_{\rm exp} ( (p')^* (E,\theta) )$ is obtained by gluing the local data
\begin{align*}
    (F^*_{R,1} p'^* E_i, \ \nabla_{can,i} + \zeta_{R,i}(F^*_{R,1} p'^* \theta_i))_{i \in I}
\end{align*}
via
\begin{align*}
    \overline{\rm exp}( h_{R,ij} (F^*_{R,1} p'^* \theta_{ij}) ).
\end{align*}
Clearly,
\begin{align*}
    p^* F^*_{Y,1} E_i = F^*_{R,1} p'^* E_i.
\end{align*}
Then, by \eqref{eq_etale_mor}, we have
\begin{align*}
    p^* \zeta_{Y,i} (F^*_{Y,1} \theta_i) = \zeta_{R,i} (F^*_{R,1} p'^* \theta_i )
\end{align*}
and
\begin{align*}
    p^* h_{Y,ij} (F^*_{Y,1} \theta_{ij}) = h_{R,ij} (F^*_{R,1} p'^* \theta_{ij} ).
\end{align*}
Therefore,
\begin{align*}
    (p^* F^*_{Y,1} E_i, \ \nabla_{can,i} + p^* \zeta_{Y,i}(F^*_{Y,1} \theta_i)) = (F^*_{R,1} p'^* E_i, \ \nabla_{can,i} + \zeta_{R,i}(F^*_{R,1} p'^* \theta_i))
\end{align*}
and
\begin{align*}
    \overline{\rm exp}( p^* h_{Y,ij} (F^*_{Y,1} \theta_{ij}) ) = \overline{\rm exp}( h_{R,ij} (F^*_{R,1} p'^* \theta_{ij}) ).
\end{align*}
This finishes the proof for the claim.

The construction of the Cartier transform is exactly the same and we leave it for the reader.
\end{proof}

\subsection{G-bundles on Root Stacks}\label{subsect_G_stack}
We consider a special class of tame smooth Deligne--Mumford stacks studied in \cite{Bor07,MO05,Cad07,Sim11}, which are called \emph{root stacks} in this paper. We first review an existence theorem about such root stacks.

\begin{thm}[Theorem 4.1 in \cite{MO05}]\label{thm_root_stack}
    Let $X$ be a smooth variety with a reduced normal crossing divisor $D=\sum\limits_{i=1}^s D_i$. We equip each irreducible component $D_i$ with a positive integer $d_i$ coprime to $p$. Denote by $\boldsymbol{d}=\{d_i,1 \leq i \leq s\}$ the collection of these integers. Then there exists a tame smooth Deligne--Mumford stack $\pi: \mathscr{X} \rightarrow X$ together with a reduced normal crossing divisor $\widetilde{D} = \sum\limits_{i=1}^s \widetilde{D}_i \subseteq \mathscr{X}$ such that
    \begin{itemize}
    \item the coarse moduli space of $\mathscr{X}$ is $X$,
    \item the map $\pi$ is finite flat and induces an isomorphism over $X \backslash D$,
    \item the pullback $\pi^* \mathcal{O}_X(-D_i)$ is equal to $\mathcal{O}_{\mathscr{X}}(-d_i \widetilde{D_i})$.
    \end{itemize}
    This stack is denoted by $\mathscr{X}_{(X,D,\boldsymbol{d})}$ and is called a root stack.
\end{thm}

\begin{rem}\label{rem_Kaw_lem}
By Kawamata covering lemma \cite[Theorem 17]{Kaw81}, the data $(X,D,\boldsymbol{d})$ determines a Galois covering $\pi: Y \rightarrow X$ with Galois group $\Gamma$, and thus $\mathscr{X}_{(X,D,\boldsymbol{d})} \cong [Y/ \Gamma]$. Note that although we work on positive characteristic, the lemma is still true under the tameness condition (see \cite[\S 12]{BP22} for instance). Under this isomorphism, we regard $G$-bundles on $\mathscr{X}_{(X,D,\boldsymbol{d})}$ as $\Gamma$-equivariant $G$-bundles on $Y$, and we use the terminology $(\Gamma,G)$-bundles on $Y$ for convenience.
\end{rem}

\begin{exmp}\label{exmp_local_stack}
    The example has been discussed in \cite[Lemma 4.3]{MO05}, and we include it here because it is very helpful to understand the local structure around a point $\widetilde{x} \in \widetilde{D} \subseteq \mathscr{X}$. Let $X=\Spec \, k[z_1,\dots,z_d]$ and $D_i = Z (z_i)$ for $1 \leq i \leq s$, and we suppose that $s \leq d$. Given a collection of positive integers $\boldsymbol{d} = \{d_i, 1 \leq i \leq s\}$, we define a natural morphism
    \begin{align*}
        k[z_1,\dots,z_d] \longrightarrow k[w_1,\dots,w_d]
    \end{align*}
    such that
    \begin{align*}
        z_i \mapsto w_i^{d_i}, \text{ } 1 \leq i \leq s, \quad z_i \mapsto w_i, \text{ } i >s.
    \end{align*}
    Then, the root stack $\mathscr{X}_{(X,D,\boldsymbol{d})}$ given in Theorem \ref{thm_root_stack} is
    \begin{align*}
        \mathscr{X}_{(X,D,\boldsymbol{d})} \cong [\Spec(k[w_1,\dots,w_d]) / \mu_{d_1} \times \cdots \times \mu_{d_s}],
    \end{align*}
    where $\mu_{d_i}$ is the group scheme of $d_i$-th roots of unity.
\end{exmp}

In the following, if there is no ambiguity, we use the notation $\mathscr{X}$ for the root stack $\mathscr{X}_{(X,D,\boldsymbol{d})}$ given in Theorem \ref{thm_root_stack}. With the same notation as above, let $\eta_i$ be the generic point of $D_i$. We define a discrete valuation ring $R_i := \mathcal{O}_{X,\eta_i}$ with fraction field $K_i$. Denote by $\widehat{R}_i$ and $\widehat{K}_i$ the corresponding completions. More generally, if $R$ is a local ring, we always use the notation $\widehat{R}$ for the completion with $\widehat{K}$ the fraction field. There exists a discrete valuation ring $\widetilde{R}_i$ such that
\begin{align*}
\mathscr{X} \times_X \Spec \, R_i \cong [\Spec \, \widetilde{R}_i / \mu_{d_i}].
\end{align*}
Furthermore, let $w_i$ (resp. $z_i$) be a generator of the maximal ideal of $\widetilde{R}_i$ (resp. $R_i$), and then the isomorphism $\mathscr{X} \times_X \Spec \, R_i \cong [\Spec \, \widetilde{R}_i / \mu_{d_i}]$ is induced by a morphism $R_i \rightarrow \widetilde{R}_i$ such that $z_i$ is sent to $w_i^{d_i}$. Based on this property, we give the definition of \emph{type} of $G$-bundles $\mathscr{E}$ on $\mathscr{X}$. The scheme $\mathscr{E}_i: = \mathscr{E} \times_{\mathscr{X}} [\Spec \, \widetilde{R}_i / \mu_{d_i}]$ can be regarded as a $(\mu_{d_i},G)$-bundle on $\Spec \, \widetilde{R}_i$. Since $\widetilde{R}_i$ is a discrete valuation ring, a $\mu_{d_i}$-action on $\mathscr{E}_i$ is given by a representation $\rho_i: \mu_{d_i} \rightarrow T$ (see \cite[\S 2.2.4]{BS15} or \cite[\S 5]{BP22} for instance).
\begin{defn}
    Given a collection of representations $\boldsymbol\rho = \{\rho_i: \mu_{d_i} \rightarrow T, 1 \leq i \leq s\}$, a $G$-bundle $\mathscr{E}$ on $\mathscr{X}$  is of \emph{type $\boldsymbol\rho$} if the $\mu_{d_i}$-action on $\mathscr{E}_i$ is given by $\rho_i$.
\end{defn}

\begin{rem}\label{rem_local_stack}
Now we consider the local structure. A $G$-bundle $\mathscr{E}$ on $\mathscr{X}$ gives a tuple $(\mathscr{E}_0, \widehat{\mathscr{E}}_i, \widehat{\Theta}_i)$, where
\begin{align*}
    \mathscr{E}_0:=\mathscr{E} \times_{\mathscr{X}} \mathscr{X} \backslash \widetilde{D}, \quad \widehat{\mathscr{E}}_i := \mathscr{E}_i \times _{ [\Spec \, \widetilde{R}_i / \mu_{d_i}] } [ \Spec \, \widehat{\widetilde{R}}_i / \mu_{d_i} ]
\end{align*}
and
\begin{align*}
    \widehat{\Theta}_i:  \mathscr{E}_0 \times_{ \mathscr{X}\backslash \widetilde{D}} [ \Spec \, \widehat{\widetilde{K}}_i / \mu_{d_i} ] \xrightarrow{\cong} \widehat{\mathscr{E}}_i \times_{ [\Spec \, \widehat{\widetilde{R}}_i / \mu_{d_i}] } [ \Spec \, \widehat{\widetilde{K}}_i / \mu_{d_i} ]
\end{align*}
is an isomorphism for each $i$. We want to remind the reader that such a tuple $(\mathscr{E}_0, \widehat{\mathscr{E}}_i, \widehat{\Theta}_i)$ cannot determine a $G$-bundle on $\mathscr{X}$ in the higher dimensional case.

The isomorphism $\widehat{\Theta}_i$ can be equivalently regarded as an element in $G(\widehat{\widetilde{K}}_i)$ such that
\begin{align*}
    \widehat{\Theta}_i (\gamma_i \cdot w_i) = \rho_i (\gamma_i) \widehat{\Theta}_i(w_i).
\end{align*}
Denote by $G(\widehat{\widetilde{K}}_i)^{\rho_i} \subseteq G(\widehat{\widetilde{K}}_i)$ the subset
\begin{align*}
    G(\widehat{\widetilde{K}}_i)^{\rho_i} := \{ \widehat{\Theta}_i \in G(\widehat{\widetilde{K}}_i) \, | \,  \widehat{\Theta}_i (\gamma_i \cdot w_i) = \rho_i (\gamma_i) \widehat{\Theta}_i(w_i)  \}.
\end{align*}
Similarly, we define a subset $G(\widehat{\widetilde{R}}_i)^{\rho_i} \subseteq G(\widehat{\widetilde{R}}_i)$ as follows
\begin{align*}
    G(\widehat{\widetilde{R}}_i)^{\rho_i} := \{ \widehat{\Psi}_i \in G(\widehat{\widetilde{R}}_i) \, | \, \rho_i (\gamma_i) \widehat{\Psi}_i(w_i) \rho_i (\gamma_i)^{-1} = \Psi_i (\gamma_i \cdot w_i) \}.
\end{align*}

In the case of stacky curves, $\widetilde{D}$ is a set of distinct points. Isomorphisms
\begin{align*}
\widehat{\Theta}_i:  \mathscr{E}_0 \times_{ \mathscr{X}\backslash \widetilde{D}} [ \Spec \, \widehat{\widetilde{K}}_i / \mu_{d_i} ] \xrightarrow{\cong} \widehat{\mathscr{E}}_i \times_{ [\Spec \, \widehat{\widetilde{R}}_i / \mu_{d_i}] } [ \Spec \, \widehat{\widetilde{K}}_i / \mu_{d_i} ]
\end{align*}
give all information of a $G$-bundle $\mathscr{E}$, and isomorphism classes of $G$-bundles of type $\boldsymbol\rho$ on $\mathscr{X}$ are in one-to-one correspondence with the double coset
\begin{align*}
\left[ G(k[X \backslash D]) \backslash \prod_{i=1}^s G(\widehat{\widetilde{K}}_i)^{\rho_{d_i}}  / \prod_{i=1}^s G(\widehat{\widetilde{R}}_i)^{\rho_{d_i}} \right],
\end{align*}
and this is an adelic description of $G$-bundles in the case of curves \cite[\S 3]{BS15}. This result holds in higher dimensional case when the normal crossing divisor $D$ is smooth.
\end{rem}

\subsection{Logarithmic G-Higgs Bundles and Logarithmic Flat G-bundles}\label{subsect_log_G_Higgs_and_conn}

\begin{defn}
    A \emph{logarithmic $G$-Higgs bundle} on $\mathscr{X}$ is a pair $(\mathscr{E},\theta)$, where $\mathscr{E}$ is a $G$-bundle on $\mathscr{X}$ and $\theta: \mathscr{X} \rightarrow \mathscr{E}(\mathfrak{g}) \otimes \Omega_{\mathscr{X}}({\rm log} \, \widetilde{D})$ is integrable. A \emph{logarithmic flat $G$-bundle} on $\mathscr{X}$ is a pair $(\mathscr{V},\nabla)$, where $\mathscr{V}$ is a $G$-bundle $\mathscr{X}$ and $\nabla: \mathcal{O}_{\mathscr{V}} \rightarrow \mathcal{O}_{\mathscr{V}} \otimes \Omega_{\mathscr{X}}({\rm log}\, \widetilde{D} )$ is an integrable $G$-connection.
\end{defn}

Let ${\rm HIG}_{p-1}(\mathscr{X}_{\rm log}/k,G)$ be the category of nilpotent logarithmic $G$-Higgs bundles on $\mathscr{X}$ of exponent $\leq p-1$ on $\mathscr{X}$. For nilpotent logarithmic flat $G$-bundles, note first that we have a short exact sequence
\begin{align*}
    0 \longrightarrow \Omega_{\mathscr{X}} \longrightarrow \Omega_{\mathscr{X}} ({\rm log}\, \widetilde{D}) \xrightarrow{\oplus {\rm res}_{\widetilde{D}_i}} \bigoplus_i \mathcal{O}_{\widetilde{D}_i} \longrightarrow 0.
\end{align*}
Given a logarithmic flat $G$-bundle $(\mathscr{V},\nabla)$, we define
\begin{align*}
    {\rm Res}_{\widetilde{D}_i} \nabla: \mathcal{O}_{\mathscr{V}} \otimes \mathcal{O}_{\widetilde{D}_i} \rightarrow \mathcal{O}_{\mathscr{V}} \otimes \mathcal{O}_{\widetilde{D}_i}
\end{align*}
as the composite
\begin{align*}
    \mathcal{O}_{\mathscr{V}} \xrightarrow{\nabla} \mathcal{O}_{\mathscr{V}} \otimes \Omega_{\mathscr{X}}({\rm log} \, \widetilde{D}) \xrightarrow{ {\rm id} \otimes {\rm res}_{\widetilde{D}_i}  } \mathcal{O}_{\mathscr{V}} \otimes \mathcal{O}_{\widetilde{D}_i}.
\end{align*}
Let ${\rm MIC}_{p-1}(\mathscr{X}_{\rm log}/k,G)$ be the category of nilpotent logarithmic flat $G$-bundles of exponent $\leq p-1$ on $\mathscr{X}$ such that ${\rm Res}_{\widetilde{D}_i} \nabla$ is also nilpotent of exponent $\leq p-1$ for every $1 \leq i \leq s$.

\begin{rem}
Since the objects we consider here are $G$-bundles, the ``tor condition" considered in \cite[Appendix]{LSYZ19} is satisfied automatically, i.e.
\begin{align*}
\mathcal{T}or_1(\mathcal{O}_{\mathscr{E}}, (F_1)_* \mathcal{O}_{\widetilde{D}_i}) = 0, \, \mathcal{T}or_1(\mathcal{O}_{\mathscr{V}}, \mathcal{O}_{\widetilde{D}_i}) = 0,
\end{align*}
where $\mathscr{E}$ is a $G$-bundle on $\mathscr{X}'$ and $\mathscr{V}$ is a $G$-bundle on $\mathscr{X}$.
\end{rem}

\begin{lem}\label{lem_HIG_MIC_stack_0}
    Suppose that $\mathscr{X}$ is $W_2(k)$-liftable. The following categories are equivalent
    \begin{center}
    \begin{tikzcd}
    {\rm HIG}_{0}(\mathscr{X}'_{\rm log}/k, G) \arrow[d, equal] \arrow[r, equal] & {\rm MIC}_{0}(\mathscr{X}_{\rm log}/k, G) \arrow[d,equal] \\
    {\rm HIG}_0 (\mathscr{X}'/k,G) \arrow[r,equal] & {\rm MIC}_0 (\mathscr{X}/k,G) \, ,
    \end{tikzcd}
    \end{center}
    where we use $``="$ for equivalence.
\end{lem}

\begin{proof}
The proof is a stacky version of \cite[Lemma A.2]{LSYZ19}. For the Higgs side, the equivalence
\begin{align*}
    {\rm HIG}_{0}(\mathscr{X}'_{\rm log}/k, G) = {\rm HIG}_{0}(\mathscr{X}'/k, G)
\end{align*}
is clear because Higgs fields are trivial in this case. Now we take an object $(\mathscr{V},\nabla) \in {\rm MIC}_{0}(\mathscr{X}_{\rm log}/k, G)$. Since $\mathcal{T}or_1(\mathcal{O}_{\mathscr{V}}, \mathcal{O}_{\widetilde{D}_i}) = 0$, we have a short exact sequence
\begin{align*}
        0 \rightarrow \mathcal{O}_{\mathscr{V}} \otimes \Omega_{\mathscr{X}} \rightarrow \mathcal{O}_{\mathscr{V}} \otimes \Omega_{\mathscr{X}}({\rm log} \, \widetilde{D} ) \xrightarrow{{\rm id} \otimes {\rm res}} \bigoplus_i \mathcal{O}_{\mathscr{V}} \otimes \mathcal{O}_{\widetilde{D}_i} \longrightarrow 0.
\end{align*}
Since the residue is nilpotent of exponent $\leq 0$, it is trivial and we have $\nabla(\mathcal{O}_{\mathscr{V}}) \subseteq \mathcal{O}_{\mathscr{V}} \otimes \Omega_{\mathscr{X}}$. Therefore,
    \begin{align*}
        {\rm MIC}_{0}(\mathscr{X}_{\rm log}/k, G) = {\rm MIC}_{0}(\mathscr{X}/k, G).
    \end{align*}
    The equivalence
    \begin{align*}
        {\rm HIG}_{0}(\mathscr{X}'/k, G) = {\rm MIC}_{0}(\mathscr{X}/k, G)
    \end{align*}
    follows from Proposition \ref{prop_cartier_des_G_stack} directly. This finishes the proof of this lemma.
\end{proof}

\begin{lem}\label{lem_claimA3}
    The following diagram commutes.
    \begin{center}
    \begin{tikzcd}
    \mathcal{O}_{\mathscr{V}} \arrow[r, "\psi"] \arrow[d] & \mathcal{O}_{\mathscr{V}} \otimes F_{1}^* \Omega_{\mathscr{X}}({\rm log} \, \widetilde{D}) \arrow[r,"{\rm id} \otimes \zeta"] & \mathcal{O}_{\mathscr{V}} \otimes \Omega_{\mathscr{X}}({\rm log} \, \widetilde{D}) \arrow[d, "{\rm id} \otimes {\rm res}_{\widetilde{D}_i}"] \\
    \mathcal{O}_{\mathscr{V}} \otimes \mathcal{O}_{\widetilde{D}_i} \arrow[rr, "-{\rm Res}_{\widetilde{D}_i}\nabla"] & & \mathcal{O}_{\widetilde{D}_i}
    \end{tikzcd}
    \end{center}
\end{lem}

\begin{proof}
    To prove this lemma, it is enough to work around a neighborhood of $\widetilde{x} \in \widetilde{D}$. By Theorem \ref{thm_root_stack} and Example \ref{exmp_local_stack}, we can choose a neighborhood around $\widetilde{x}$ such that it is isomorphic to a root stack in the form $\mathscr{X}_{(X,D,\boldsymbol{d})} \cong [\Spec(k[w_1,\dots,w_d]) / \mu_{d_1} \times \cdots \times \mu_{d_s}]$. Then, this lemma follows directly from \cite[Claim A.3]{LSYZ19}.
\end{proof}

\begin{thm}\label{thm_HIG_MIC_G_log_stack}
Suppose that $(\mathscr{X},\widetilde{D})$ is $W_2$-liftable. Then the categories ${\rm HIG}_{p-1}(\mathscr{X}'_{\rm log}/k , G) $ and ${\rm MIC}_{p-1}(\mathscr{X}_{\rm log}/k , G)$ are equivalent.
\end{thm}

\begin{proof}
    Similar to the proof of \cite[Theorem A.1]{LSYZ19}, based on Theorem \ref{thm_HIG_MIC_G_stack} and Lemma \ref{lem_HIG_MIC_stack_0}, it is enough to work on local charts, and thus we suppose that
    \begin{align*}
        \mathscr{X} = [\Spec(k[w_1,\dots,w_d]) / \mu_{d_1} \times \cdots \times \mu_{d_s}].
    \end{align*}
    Given $(\mathscr{E},\theta) \in {\rm HIG}_{p-1}(\mathscr{X}'_{\rm log}/k , G)$, we obtain a logarithmic flat $G$-bundle $(\mathscr{V},\nabla)$ via a logarithmic version of the inverse Cartier transform given in Theorem \ref{thm_HIG_MIC_G_stack}. For convenience, we regard $(\mathscr{E},\theta)$ (resp. $(\mathscr{V},\nabla)$) as equivariant logarithmic $G$-Higgs bundles $(E,\theta)$ (resp. logarithmic flat $G$-bundles $(V,\nabla)$) on $\Spec( k[w_1,\dots,w_d] )$. Thus, following the construction of \eqref{eq_new_conn} in the proof of Theorem \ref{thm_HIG_MIC_G}, we have
    \begin{align*}
        \nabla = \nabla_{can} + \zeta(F_1^* \theta),
    \end{align*}
    and then
    \begin{align*}
        {\rm Res}_{\widetilde{D}_i} \nabla = F_1^* {\rm Res}_{\widetilde{D}_i} \theta.
    \end{align*}
    This implies that the residues are nilpotent. Therefore, we have $(\mathscr{V},\nabla) \in {\rm MIC}_{p-1}(\mathscr{X}_{\rm log}/k , G)$.

    For the other direction, we take an object $(\mathscr{V},\nabla) \in {\rm MIC}_{p-1}(\mathscr{X}_{\rm log}/k , G)$. With the same proof as in \S\ref{subsect_car_G}, we obtain a triple $(\mathscr{V}',\nabla',\psi')$, where $(\mathscr{V}',\nabla')$ is a logarithmic flat $G$-bundle with trivial $p$-curvature and $\psi'$ is the $F$-Higgs field horizontal with respect to $\nabla'$. Then, the triple $(\mathscr{V}',\nabla',\psi')$ descents to a nilpotent logarithmic $G$-Higgs bundle $(\mathscr{E},\theta)$ on $\mathscr{X}'$, which is clear an object in the category ${\rm HIG}_{p-1}(\mathscr{X}'_{\rm log}/k , G)$.
\end{proof}

\begin{rem}\label{rem_type}
    In this remark, we discuss how type changes via the correspondence given in Theorem \ref{thm_HIG_MIC_G_log_stack}. Since type is a local property, we still work on the stack
    \begin{align*}
        \mathscr{X} = [\Spec(k[w_1,\dots,w_d]) / \mu_{d_1} \times \cdots \times \mu_{d_s}]
    \end{align*}
    as in Example \ref{exmp_local_stack} and consider equivariant objects. We fix some notations in this example. Let
    \begin{align*}
        \Gamma = \mu_{d_1} \times \cdots \times \mu_{d_s}, \quad Y = \Spec(k[w_1,\dots,w_d]).
    \end{align*}
    Let $E \cong Y' \times G$ be a trivial $G$-bundle of type $\boldsymbol\rho = \{\rho_i, 1 \leq i \leq s\}$ on $Y$. Thus, the $\Gamma$-action on $E$ is given as follows
    \begin{align*}
        (\gamma_1,\dots,\gamma_s) \cdot ((w_1,\dots,w_d),g) = ((\gamma_1 w_1,\dots,\gamma_s w_s, \dots, w_d), \rho_1(\gamma_1)\dots \rho_s(\gamma_s) g ).
    \end{align*}
    Under the inverse Cartier transform, we obtain a $(\Gamma,G)$-bundle $V=F_1^* E$ which is of type $\boldsymbol\rho^p = \{\rho_i^p, \, 1 \leq i \leq s\}$.
\end{rem}

\subsection{Stability Condition}\label{subsect_stab_cond_stack}
In this subsection, we make a further assumption on the root stack $\mathscr{X} \cong [Y / \Gamma]$ that it is projective, i.e. the coarse moduli space $X$ is projective. With respect to a given ample line bundle $H$ on $X$, the degree of a locally free sheaf $\mathscr{F}$ on is given in \cite[4.1.2]{Bor07} and denote it by $\deg \mathscr{F} := \deg_H \mathscr{F}$. Moreover, a big open substack $\mathscr{X}^{\rm big} \subseteq \mathscr{X}$ can be regarded as a $\Gamma$-invariant big open subset $Y^{\rm big} \subseteq Y$ such that $\mathscr{X}^{\rm big} \cong [Y^{\rm big}/ \Gamma]$. Then we define stability conditions of logarithmic $\lambda$-connections on $\mathscr{X}$ in the same way as Definition \ref{defn_stab_lambda}.

\begin{defn}
A logarithmic $\lambda$-connection $(\mathscr{E},\nabla)$ on $\mathscr{X}$ is \emph{$R$-semistable} (resp. \emph{$R$-stable}), if
\begin{itemize}
	\item for any maximal parabolic subgroup $P \subseteq G$;
	\item for any antidominant character $\chi: P \rightarrow \mathbb{G}_m$ trivial on the center;
    \item for any big open substack $\mathscr{X}^{\rm big} \subseteq \mathscr{X}$;
	\item for any $\nabla$-compatible reduction of structure group $\sigma: \mathscr{X}^{\rm big} \rightarrow (\mathscr{E}|_{ \mathscr{X}^{\rm big} } )/P$,
\end{itemize}
we have
\begin{align*}
    \deg L(\sigma,\chi) \geq 0 \text{ (resp. $> 0$)}.
\end{align*}
\end{defn}

We introduce the following categories
\begin{itemize}
    \item ${\rm HIG}^{s(s)}_{p-1}(\mathscr{X}_{\rm log}/k,G)$: the category of $R$-(semi)stable nilpotent logarithmic $G$-Higgs bundles of exponent $\leq p-1$ on $\mathscr{X}$;
    \item ${\rm MIC}^{(s)s}_{p-1}(\mathscr{X}_{\rm log}/k,G)$: the category of $R$-(semi)stable nilpotent logarithmic $G$-Higgs bundles of exponent $\leq p-1$ on $\mathscr{X}$.
\end{itemize}

\begin{thm}\label{thm_HIG_MIC_G_stab_stack}
Suppose that $(\mathscr{X},\widetilde{D})$ is $W_2(k)$-liftable. The categories ${\rm HIG}^{(s)s}_{p-1}(\mathscr{X}'_{\rm log}/k , G)$ and ${\rm MIC}^{(s)s}_{p-1}(\mathscr{X}_{\rm log}/k , G)$ are equivalent.
\end{thm}

\begin{proof}
We remind the reader that the degree of locally free sheaves on $\mathscr{X}'$ is defined by $H' := \pi^*_{X/k} H$, while the degree of locally free sheaves on $\mathscr{X}$ is given by $H$. It is easy to check that $\deg_{H'} \mathscr{F}' = p \deg_{H} \mathscr{F}$, where $\mathscr{F}$ is a locally free sheaf on $\mathscr{X}$. As discussed in Theorem \ref{thm_HIG_MIC_G_stab}, we only have to check that there is a one-to-one correspondence between $\theta$-compatible reductions of a $G$-Higgs bundle $(\mathscr{E}.\theta)$ on $\mathscr{X}'$ and $\nabla$-compatible reductions of the corresponding flat $G$-bundle $(\mathscr{V},\nabla)$ on $\mathscr{X}$ given by Theorem \ref{thm_HIG_MIC_G_log_stack}. In this proof, we start with a $\theta$-compatible reduction $\sigma$ and construct a $\nabla$-compatible reduction $\varsigma$. The other direction can be proved in a similar way and we omit the details.

    Let $(\mathscr{E},\theta) \in {\rm HIG}_{p-1}(\mathscr{X}'_{\rm log}/k , G)$ be a $G$-Higgs bundle. Under the isomorphism $\mathscr{X} \cong [Y/R]$, we regard $(\mathscr{E},\theta)$ as a tuple $(E,\theta,\tau')$, where $(E,\theta)$ is a logarithmic $G$-Higgs bundle on $Y'$ and $\tau':p_1^*(E,\theta) \cong p_2^*(E,\theta)$ is an isomorphism as in the proof of Theorem \ref{thm_HIG_MIC_G_stack}. A $\theta$-compatible reduction $\sigma: \mathscr{X}' \rightarrow \mathscr{E}/P$ can be regarded as a $\theta$-compatible reduction $\sigma: Y' \rightarrow E/P$ such that the following diagram commutes,
    \begin{center}
    \begin{tikzcd}
    p_1^* E_\sigma \arrow[r, "p_1^* \sigma"] \arrow[d, "\tau'_\sigma"] & p_1^* E \arrow[d, "\tau' "] \\
    p_2^* E_\sigma \arrow[r, "p_2^* \sigma"] & p_2^* E
    \end{tikzcd}
    \end{center}
    where $\tau'_\sigma: p_1^* E_\sigma \rightarrow p_2^* E_\sigma$ is the isomorphism induced by $\tau'$. Moreover, $p_i^* \sigma$ is compatible with the Higgs field $p_i^* \theta$. By Theorem \ref{thm_HIG_MIC_G_stab}, the triple $(E,\theta,\sigma)$ on $Y'$ corresponds to a triple $(V,\nabla,\varsigma)$, where $(V,\nabla)$ is a nilpotent logarithmic flat $G$-bundle and $\varsigma: Y \rightarrow V/P$ is a $\nabla$-compatible reduction. Similarly, $p_i^* (E,\theta,\sigma)$ corresponds to $p_i^* (V,\nabla,\varsigma)$ on $R$. The isomorphisms
    \begin{align*}
        \tau': p_1^* (E,\theta) \cong p_2^* (E,\theta), \quad \tau'_\sigma: p_1^* (E_\sigma,\theta_\sigma) \cong p_2^* (E_\sigma,\theta_\sigma)
    \end{align*}
    induce ones
    \begin{align*}
        \tau: p_1^* (V,\nabla) \cong p_2^* (V,\nabla), \quad  \tau_\sigma: p_1^*(V_\varsigma, \nabla_{\varsigma}) \cong p_2^*(V_\varsigma, \nabla_{\varsigma})
    \end{align*}
    by Theorem \ref{thm_HIG_MIC_G_stack}. Thus, we obtain a triple $(V,\nabla,\varsigma)$ together with an isomorphism $\tau$ such that the diagram
    \begin{center}
    \begin{tikzcd}
    p_1^* V_\varsigma \arrow[r, "p_1^* \varsigma"] \arrow[d, " \tau_\sigma"] & p_1^* V \arrow[d, "\tau"] \\
    p_2^* V_\varsigma \arrow[r, "p_2^* \varsigma"] & p_2^* V
    \end{tikzcd}
    \end{center}
    commutes, and then we get a nilpotent logarithmic flat $G$-bundle $(\mathscr{V},\nabla) \in {\rm MIC}_{p-1}(\mathscr{X}_{\rm log}/k , G)$ together with a $\nabla$-compatible reduction $\varsigma$ of structure group.
\end{proof}

\section{Logahoric Higgs Torsors and Logahoric Connections}\label{sect_parah}

The goal of this section is to prove a logahoric version of the nonabelian Hodge correspondence for $G$-bundles via exponential twisting. We first studied parahoric group schemes over a smooth variety $X$ (with a given reduced normal crossing divisor $D$), which is inspired by Balaji--Pandey's work \cite[Theorem 5.4]{BP22}. Then we prove a correspondence between parahoric torsors on $X$ and $G$-bundles on a root stack $\mathscr{X}$ (Proposition \ref{prop_parah_equiv_G}), which also preserves the stability condition (Proposition \ref{prop_parah_equiv_G_stab}). Based on the above correspondence and results given on root stacks (Theorem \ref{thm_HIG_MIC_G_log_stack} and \ref{thm_HIG_MIC_G_stab_stack}), we give the proof of the main result (Theorem \ref{thm_log_NAHC}).

\subsection{Parahoric Group Scheme}\label{subsect_parah}
We fix a maximal torus $T \subseteq G$ and denote by $\mathcal{R}$ the root system. Define $X(T):={\rm Hom}(T,\mathbb{G}_m)$ and $Y(T):={\rm Hom}(\mathbb{G}_m,T)$. Denote by
\begin{align*}
\langle \cdot, \cdot \rangle : Y(T) \times X(T) \rightarrow \mathbb{Z}
\end{align*}
the canonical pairing, which can be naturally extended to $\mathbb{Q}$. A \emph{weight} $\alpha$ is an element in $Y(T) \otimes_{\mathbb{Z}} \mathbb{Q}$, and let $d$ be the least common multiple of the denominators. Clearly, $d \alpha \in Y(T)$. For convenience, we say that $d$ is the denominator of $\alpha$. The weight $\alpha$ is called \emph{tame} if $d$ is coprime to $p$ \cite[\S 2]{LS21}. Given a weight $\alpha$ and a root $r \in \mathcal{R}$, the pairing $\langle \alpha , r\rangle$ is a rational number, and sometimes we use the notation $r(\alpha)$. All weights considered in this paper are assumed to be tame and satisfy the condition $r(\alpha) \leq 1$ for any $r \in \mathcal{R}$.

Let $X$ be a smooth variety over $k$ with a given reduced normal crossing divisor $D=\sum\limits_{i=1}^s D_i$, where $D_i$ are irreducible components for $1 \leq i \leq s$. We equip each $D_i$ with a tame weight $\alpha_i$, and denote by $\boldsymbol\alpha$ the collection of these weights. Furthermore, denote by $\boldsymbol{d} = \{d_i, 1 \leq i \leq s\}$ the collection of the corresponding denominators. Recall that $\eta_i$ is the generic point of $D_i$ and $R_i=\mathcal{O}_{X,\eta_i}$ with fraction field $K_i$. Denote by $\widehat{R}_i$ and (resp. $\widehat{K}_i$) the completion, which is regarded as the formal neighbourhood of the generic point $\eta_i$. Let $G_{\alpha_i}(\widehat{R}_i)$ be the parahoric group associated to a given weight $\alpha_i$
\begin{align*}
    G_{\alpha_i}(\widehat{R}_i) := \langle T(\widehat{R}_i), U_r (z_i^{m_r(\alpha_i)}\widehat{R}_i), r \in \mathcal{R} \rangle,
\end{align*}
where $U_r$ is the unipotent group associated to $r \in \mathcal{R}$, $m_r(\alpha_i) := \lceil -r(\alpha_i) \rceil$ is an integer and $\lceil \, \cdot \, \rceil$ is the ceiling function. We denote by $\mathcal{G}_{\alpha_i}$ the corresponding parahoric group scheme on $\Spec \, \widehat{R}_i$ with Lie algebra $\mathfrak{g}_{\alpha_i}$ \cite[\S 2 and \S 3]{BS15}. Now we give the definition of parahoric group schemes on higher dimensional varieties.

\begin{defn}\label{defn_parah_grp_sch}
    A smooth group scheme $\mathcal{G}$ on $X$ is \emph{parahoric} if
    \begin{align*}
        \mathcal{G} \times_X X\backslash D \cong G \times X\backslash D
    \end{align*}
    and
    \begin{align*}
        \mathcal{G} \times_X \Spec\, \widehat{R}_i \cong \mathcal{G}_{\alpha_i}
    \end{align*}
    for some tame weights $\alpha_i$ for $1 \leq i \leq s$.
\end{defn}

The idea of the above definition comes from a recent work by Balaji and Pandey \cite{BP22}, especially \cite[Theorem 5.4]{BP22}, where the authors constructed parahoric group schemes on higher dimensional varieties for simply connected and almost simple algebraic group $G$. In the following, we will show that parahoric group schemes in the above sense exist.

Let $\mathscr{X} := \mathscr{X}_{(X,D ,\boldsymbol{d})}$ be a root stack with coarse moduli space $X$ studied in \S\ref{subsect_G_stack}. Let $\mathscr{F}$ be a $G$-bundle on $\mathscr{X}$. Consider the stack of automorphisms ${\rm Aut}_{\mathscr{X}}(\mathscr{F},\mathscr{F})$. We define a functor
\begin{align*}
   (\pi_* {\rm Aut}_{\mathscr{X}}(\mathscr{F},\mathscr{F}) ) (T) := {\rm Hom}_{\mathscr{X}} (T \times_X \mathscr{X},  {\rm Aut}_{\mathscr{X}}(\mathscr{F},\mathscr{F}))
\end{align*}
for every $X$-scheme $T$.

\begin{lem}\label{lem_parah_grp}
Let $\mathscr{F}$ be a $G$-bundle of type $\boldsymbol\rho$ on $\mathscr{X}$ such that $\mathscr{F}|_{X \backslash D}$ is a trivial $G$-bundle. The functor $(\pi_* {\rm Aut}_{\mathscr{X}}(\mathscr{F},\mathscr{F}) )$ is representable by a parahoric group scheme $\mathcal{G}$ on $X$ such that $\mathcal{G} \times_X \Spec\, \widehat{R}_i \cong \mathcal{G}_{\alpha_i}$ for $1 \leq i \leq s$, where the tame weight $\alpha_i$ is determined by $\rho_i$.
\end{lem}

\begin{proof}
We prove this lemma under the isomorphism $\mathscr{X} \cong [Y / \Gamma]$ and regard $\mathscr{F}$ as a $(\Gamma,G)$-bundle on $Y$ (Remark \ref{rem_Kaw_lem}). Thus the functor $\pi_* {\rm Aut}_{\mathscr{X}}(\mathscr{F},\mathscr{F})$ is equivalent to $(\pi_* {\rm Aut}_{Y}(\mathscr{F},\mathscr{F}))^{\Gamma}$. Following \cite[\S 7.6]{BLR90} (see also \cite[\S 4]{BS15} and \cite{Dam24}), the functor $(\pi_* {\rm Aut}_{Y}(\mathscr{F},\mathscr{F}))^{\Gamma}$ is representable by a smooth group scheme $\mathcal{G}$ on $X$.

Now we will show that the group scheme $\mathcal{G}$ is parahoric. Note that ``parahoric" is a local condition, and thus it is enough to work on $[\Spec \, \widetilde{R}_i / \mu_{d_i}]$ for $1 \leq i \leq s$ (see notations in \S\ref{subsect_G_stack}). Given a representation $\rho_i: \mu_{d_i} \rightarrow T$, we can associate it with a tame weight $\alpha_i$ \cite[Lemma 2.2.8]{BS15}. Then, by \cite[Theorem 2.3.1]{BS15}, we have
\begin{align*}
    \mathcal{G} \times_{X} \Spec \, \widehat{R}_i =  \pi_* {\rm Aut}_{ \Spec \, \widehat{ \widetilde{R}}_i  } (\mathscr{F}|_{\Spec \, \widehat{ \widetilde{R}}_i  } , \mathscr{F}|_{\Spec \, \widehat{ \widetilde{R}}_i } )^{\Gamma} \cong \mathcal{G}_{\alpha_i}.
\end{align*}
Therefore, the group scheme $\mathcal{G}$ on $X$ is parahoric.
\end{proof}

\begin{notn}\label{notn_parah_grp}
Lemma \ref{lem_parah_grp} gives a way to construct parahoric group schemes over higher dimensional varieties, which depends on a $G$-bundle on the root stack. Thus, it is meaningful to use the notation $\mathcal{G}_{\mathscr{F},\boldsymbol\alpha}$ for the parahoric group scheme constructed in the lemma, where $\mathscr{F}$ is the given $G$-bundle on $\mathscr{X}$ and $\boldsymbol\alpha$ is the collection of tame weights determined by the type of $\mathscr{F}$. However, if there is no ambiguity, we still use the notation $\mathcal{G}_{\boldsymbol\alpha}:=\mathcal{G}_{\mathscr{F},\boldsymbol\alpha}$ for simplicity.
\end{notn}

\begin{rem}\label{rem_reduc_corr}
Balaji and Pandey worked on simply connected and almost simple algebraic group and constructed parahoric group schemes on smooth projective varieties (with reduced normal crossing divisors) based on the existence of a Lie algebra bundle \cite[Theorem 5.4]{BP22}. In our setup, Lemma \ref{lem_parah_grp} provides another approach to construct parahoric group schemes when $G$ is reductive. However, we do not know whether all such parahoric group schemes (Definition \ref{defn_parah_grp_sch}) can be constructed in this way. Nevertheless, if the reduced normal crossing divisor is smooth, all parahoric group schemes (Definition \ref{defn_parah_grp_sch}) can be constructed in this way because this case is exactly the same as that on curves and we can use transition functions to give the construction (see \cite[\S 3]{BS15} or Remark \ref{rem_local_stack}).
\end{rem}

\subsection{Parahoric Torsors on $X$ vs. Principal Bundles on $\mathscr{X}$}\label{subsect_parah_equiv}

We follow the same setup as in \S\ref{subsect_parah}. Let $\mathscr{F}$ be a $G$-bundle of type $\boldsymbol\rho$ on $\mathscr{X}$ such that $\mathscr{F}|_{X \backslash D}$ is trivial, and denote by $\mathcal{G}_{\boldsymbol\alpha}$ the corresponding parahoric group scheme given by Lemma \ref{lem_parah_grp}.

\begin{prop}\label{prop_parah_equiv_G}
    The category ${\rm Bun}(X, \mathcal{G}_{\boldsymbol\alpha})$ of parahoric $\mathcal{G}_{\boldsymbol\alpha}$-torsors on $X$ is equivalent to the category ${\rm Bun}(\mathscr{X},G,\boldsymbol\rho)$ of $G$-bundles of type $\boldsymbol\rho$ on $\mathscr{X}$.
\end{prop}

\begin{proof}
    For convenience, we regard $\mathscr{X}$ as a global quotient $[Y/ \Gamma]$, and then the proposition follows directly from \cite[Theorem 11]{Dam24}. The equivalence is given by
    \begin{equation*}
    \begin{split}
    \beta_{\mathscr{F}}: & {\rm Bun}(\mathscr{X},G,\boldsymbol\rho) \rightarrow {\rm Bun}(X, \mathcal{G}_{\boldsymbol\alpha})\\
    &\mathscr{F}^\prime\mapsto \pi_* {\rm Aut}_{\mathscr{X}}(\mathscr{F},\mathscr{F}^\prime)
    \end{split}
    \end{equation*}
     and
     \begin{equation*}
    \begin{split}
    \eta_{\mathscr{F}}:&{\rm Bun}(X, \mathcal{G}_{\boldsymbol\alpha}) \rightarrow {\rm Bun}(\mathscr{X},G,\boldsymbol\rho) \\
    &\mathscr{E}\mapsto \pi^*\mathscr{E}\times_{\pi^*\mathcal{G}_{\boldsymbol\alpha}}\mathscr{F}.
    \end{split}
    \end{equation*}
\end{proof}

Similar to the local discussion for $G$-bundles on Deligne--Mumford stacks in Remark \ref{rem_local_stack}, we consider the local picture of parahoric $\mathcal{G}_{\boldsymbol\alpha}$-torsors. Given a parahoric $\mathcal{G}_{\boldsymbol\alpha}$-torsor $\mathcal{E}$, we associate it with a tuple $(\mathcal{E}_0,\widehat{\mathcal{E}}_i,\widehat{\Theta}'_i)$, where
\begin{align*}
    \mathcal{E}_0:= \mathcal{E} \times_X X \backslash D, \quad \widehat{\mathcal{E}}_i := \mathcal{E} \times_X \Spec \, \widehat{R}_i
\end{align*}
and
\begin{align*}
    \widehat{\Theta}'_i: \mathcal{E}_0 \times_{X \backslash D} \Spec \, \widehat{K}_i \rightarrow \widehat{\mathcal{E}}_i \times_{\Spec \, \widehat{R}_i} \Spec \, \widehat{K}_i
\end{align*}
is an isomorphism for each $i$. In this case, the isomorphism $\widehat{\Theta}'_i$ can be regarded as an element in $G(\widehat{K}_i)$.

\begin{rem}\label{rem_trans_func}
In this remark, we briefly explain Proposition \ref{prop_parah_equiv_G} in the viewpoint of changing transition functions. We first recall some notations. Let $R_i$ be the discrete valuation rings given by the generic point $\eta_i$ of $D_i$ with fraction field $K_i$. We have
\begin{align*}
\mathscr{X} \times_X \Spec \, R_i \cong [\Spec \, \widetilde{R}_i / \mu_{d_i}].
\end{align*}

Let $\mathscr{E}$ be a $G$-bundle of type $\boldsymbol\rho$ on $\mathscr{X}$ and denote by $\mathcal{E}$ the corresponding parahoric $\mathcal{G}_{\boldsymbol\alpha}$-torsor on $X$. Denote by $(\mathscr{E}_0, \widehat{\mathscr{E}}_i, \widehat{\Theta}_i)$ (resp. $(\mathcal{E}_0, \widehat{\mathcal{E}}_i, \widehat{\Theta}'_i)$) the tuple associated to $\mathscr{E}$ (resp. $\mathcal{E}$). We will explain the relation between $\widehat{\Theta}_i$ and $\widehat{\Theta}'_i$ as follows. Recall that
\begin{align*}
    \widehat{\Theta}_i (\gamma_i \cdot w_i) = \rho_i(\gamma_i) \widehat{\Theta}_i(w_i).
\end{align*}
By \cite[(2.2.9.4)]{BS15}, we can choose an element $\Delta_i \in T(\widehat{\widetilde{R}}_i)$ such that
\begin{align*}
    \Delta_i (\gamma_i \cdot w_i)  = \rho_i(\gamma_i)\Delta_i(w_i),
\end{align*}
and define new isomorphisms
\begin{align*}
    \Delta^{-1}_i \widehat{\Theta}_i, \ 1 \leq i \leq s.
\end{align*}
It is easy to check that $\Delta^{-1}_i \widehat{\Theta}_i$ is $\mu_{d_i}$-invariant, i.e.
\begin{align*}
    (\Delta^{-1}_i \widehat{\Theta}_i) (\gamma_i \cdot w_i) = (\Delta^{-1}_i \widehat{\Theta}_i)(w_i).
\end{align*}
Thus, $\Delta^{-1}_i \widehat{\Theta}_i$ descents to a well-defined element in $G(\widehat{K}_i)$ under the substitution $z_i = w_i^{d_i}$, which is exactly $\widehat{\Theta}'_i$ (up to conjugation). Moreover, there is a one-to-one correspondence between elements in $G(\widehat{\widetilde{K}}_i)^{\rho_i}$ and elements in $G(\widehat{K}_i)$ in the above sense.

At the end of this remark, we briefly explain the relation between tuples $(\mathcal{E}_0, \widehat{\mathcal{E}}_i, \widehat{\Theta}'_i)$ (resp. $(\mathscr{E}_0, \widehat{\mathscr{E}}_i, \widehat{\Theta}_i)$) and parahoric $\mathcal{G}_{\boldsymbol\alpha}$-torsors $\mathcal{E}$ on $X$ (resp. $G$-bundles on $\mathscr{X}$). We take parahoric torsors as an example and the approach given here is also included in \cite[first and second paragraph of the proof of Theorem 5.4]{BP22}. We start with a tuple $(\mathcal{E}_0, \widehat{\mathcal{E}}_i, \widehat{\Theta}'_i)$. By \cite[\S 6]{BLR90}, the tuple induces one $(\mathcal{E}_0, \mathcal{E}_i, \Theta'_i)$, where $\mathcal{E}_i$ is a scheme of finite type on $\Spec \, R_i$ and
\begin{align*}
    \Theta'_i : \mathcal{E}_0 \times_{X \backslash D} \Spec \, K_i \rightarrow \mathcal{E}_i \times_{\Spec \, R_i} \Spec \, K_i
\end{align*}
is an isomorphism. Therefore, the tuple $(\mathcal{E}_0, \mathcal{E}_i, \Theta'_i)$ gives a scheme $\mathcal{E}^{\rm big}$ over an open subset $X^{\rm big}$ of $X$ such that $\mathcal{E} \times_X X^{\rm big} \cong \mathcal{E}^{\rm big}$ and ${\rm codim} (X \backslash X^{\rm big}) \geq 2$. Moreover, $\mathcal{E}^{\rm big}$ is a parahoric $\mathcal{G}^{\rm big}_{\boldsymbol\alpha}$-torsor, where $\mathcal{G}^{\rm big}_{\boldsymbol\alpha}:= \mathcal{G}_{\boldsymbol\alpha} \times_X X^{\rm big}$.
\end{rem}

\begin{rem}\label{rem_Higgs_field}
In this remark, we mention a property, which will be used to study how Higgs fields and connections change later. Recall that we define
\begin{align*}
    G(\widehat{\widetilde{R}}_i)^{\rho_i} := \{ \widehat{\Psi}_i \in G(\widehat{\widetilde{R}}_i) \, | \, \rho_i (\gamma_i) \widehat{\Psi}_i(w_i) \rho_i (\gamma_i)^{-1} = \Psi_i (\gamma_i \cdot w_i) \}.
\end{align*}
in Remark \ref{rem_local_stack}. With a similar approach as \cite[Theorem 2.3.1]{BS15}, we have
\begin{align*}   G(\widehat{\widetilde{R}}_i)^{\rho_{d_i}} \cong G_{\alpha_i} (\widehat{R}_i).
\end{align*}
More precisely, given an isomorphism $\widehat{\Psi}_i: \mathscr{E}_i \rightarrow \mathscr{E}_i$, we regard it as an element in $G(\widehat{\widetilde{R}}_i)^{\rho_{i}}$, i.e.
\begin{align*}
    \rho(\gamma_i) \widehat{\Psi}_i(w_i) \rho(\gamma_i)^{-1} = \widehat{\Psi}_i (\gamma_i \cdot w_i).
\end{align*}
Define $\widehat{\Psi}'_i = \Delta_i^{-1} \widehat{\Psi}_i \Delta_i$, which is $\mu_{d_i}$-invariant, i.e.
\begin{align*}
    \widehat{\Psi}'_i(\gamma_i \cdot w_i)  = \widehat{\Psi}'_i(w_i).
\end{align*}
Thus, $\widehat{\Psi}'_i(w_i)$ descents to a well-defined element in $G(\widehat{K}_i)$ and it is easy to check that it is in $G_{\alpha_i}(\widehat{R}_i)$. Indeed, conjugating by $\Delta_i$ and substituting $z_i = w_i^{d_i}$ induces the desired isomorphism.
\end{rem}

Now we suppose that $X$ is projective, and then $Y$ given in Remark \ref{rem_Kaw_lem} can be required to be projective by Kawamata covering lemma. Let $\eta$ (resp. $\widetilde{\eta}$) be the generic point of $X$ (resp. $Y$). Let $\mathcal{E}$ be parahoric $\mathcal{G}_{\boldsymbol\alpha}$-torsor on $X$. Let $\mathcal{P}$ be a maximal parabolic subgroup of the generic fiber $(\mathcal{G}_{\boldsymbol\alpha})_\eta$ of $\mathcal{G}_{\boldsymbol\alpha}$. Given a generic reduction $\varsigma_\eta: X_\eta \rightarrow \mathcal{E}_\eta / \mathcal{P}$ and a character $\kappa: \mathcal{P} \rightarrow \mathbb{G}_{m,\eta}$, we obtain a line bundle $L(\varsigma_\eta,\kappa)$ on $X \backslash D$.

\begin{lem}\label{lem_parab_line_bundle}
This line bundle $L(\varsigma_\eta,\kappa)$ can be extended canonically to a parabolic line bundle on a big open subset $X^{\rm big}$ and denote it by $L(\varsigma,\kappa)$.
\end{lem}

\begin{proof}
By Proposition \ref{prop_parah_equiv_G}, denote by $E$ the $(\Gamma,G)$-bundle on $Y$ corresponding to $\mathcal{E}$. With the same argument as \cite[Proposition 6.3.1]{BS15}, we have the following correspondences:
\begin{itemize}
    \item maximal parabolic subgroups $\mathcal{P} \subseteq (\mathcal{G}_{\boldsymbol\alpha})_\eta$ and maximal parabolic subgroups $P$ of $G$;
    \item characters $\mathcal{P} \rightarrow \mathbb{G}_{m,\eta}$ and characters $P \rightarrow \mathbb{G}_m$;
    \item generic reductions $\varsigma_\eta: X_\eta \rightarrow \mathcal{E}_\eta / \mathcal{P}$ and $\Gamma$-equivariant generic reductions $\sigma_{\widetilde{\eta}}: Y_{\widetilde{\eta}} \rightarrow E_{\widetilde{\eta}} / P_{\widetilde{\eta}}$.
\end{itemize}
Given the above correspondences, the corresponding $\Gamma$-equivariant generic reduction of structure group $\sigma_{\widetilde{\eta}}: Y_{\widetilde{\eta}} \rightarrow E_{\widetilde{\eta}} / P_{\widetilde{\eta}}$ is well-defined on an open subset $Y^{\rm big} \subseteq Y$ by \cite[Theorem 12.60]{GW20}, where ${\rm codim}(Y \backslash Y^{\rm big}) \geq 2$, and denote the corresponding reduction of structure group on $Y^{\rm big}$ by $\sigma$. Then we obtain a $\Gamma$-equivariant line bundle $L(\sigma,\chi)$ on $Y^{\rm big}$. By \cite{Bor07}, $L(\sigma,\chi)$ corresponds to a parabolic line bundle on a big open subset $X^{\rm big} := Y^{\rm big} / \Gamma$, which is the desired parabolic bundle.
\end{proof}

As a parabolic line bundle, denote by $par \deg L(\varsigma,\kappa)$ the parabolic degree of $L(\varsigma,\kappa)$. Following the above construction, we define the $R$-stability condition for parahoric torsors as follows:
\begin{defn}\label{defn_stab_cond_parah}
A parahoric $\mathcal{G}_{\boldsymbol\alpha}$-torsor $\mathcal{E}$ is \emph{$R$-semistable} (resp. \emph{$R$-stable}), if
\begin{itemize}
	\item for any maximal parabolic subgroup $\mathcal{P} \subseteq (\mathcal{G}_{\boldsymbol\alpha})_\eta$;
	\item for any antidominant character $\kappa: \mathcal{P} \rightarrow \mathbb{G}_{m,\eta}$ trivial on the center;
	\item for any generic reduction of structure group $\varsigma_\eta: X_\eta \rightarrow \mathcal{E}_\eta / \mathcal{P}$,
\end{itemize}
we have
\begin{align*}
par \deg L(\varsigma,\kappa) \geq 0 \text{ (resp. $> 0$)}.
\end{align*}
\end{defn}

\begin{prop}\label{prop_parah_equiv_G_stab}
Let $\mathcal{E}$ be a parahoric $\mathcal{G}_{\boldsymbol\alpha}$-torsor on $X$. Denote by $\mathscr{E}$ the corresponding $G$-bundle of type $\boldsymbol\rho$ on $\mathscr{X}$. Then $\mathcal{E}$ is $R$-(semi)stable if and only if $\mathscr{E}$ is $R$-(semi)stable.
\end{prop}

\begin{proof}
Under the isomorphism $\mathscr{X} \cong [Y/ \Gamma]$, the proof of Lemma \ref{lem_parab_line_bundle} implies the following one-to-one correspondences
\begin{itemize}
\item characters: $\chi: P \rightarrow \mathbb{G}_m$ and $\kappa: P_{\eta} \rightarrow \mathbb{G}_{m,\eta}$,
\item reductions: $\varsigma_\eta: X_\eta \rightarrow \mathcal{E}_\eta / \mathcal{P}$ and reductions of $\sigma: \mathscr{X}^{\rm big} \rightarrow ( \mathscr{E} |_{\mathscr{X}^{\rm big}}) /P$ for some big open substack $\mathscr{X}^{\rm big}$.
\end{itemize}
Therefore, we only have to show the equality
\begin{align*}
    \deg L(\sigma, \chi) = par \deg L(\varsigma,\kappa),
\end{align*}
which is given in \cite[Theorem 4.3]{Bor07}.
\end{proof}

\subsection{Logahoric Higgs Torsors and Logahoric Connections}\label{subsect_parah_Higgs_conn}

\begin{defn}
A \emph{logahoric $\mathcal{G}_{\boldsymbol\alpha}$-Higgs torsor} is a pair $(\mathcal{E},\vartheta)$, where $\mathcal{E}$ is a parahoric $\mathcal{G}_{\boldsymbol\alpha}$-torsor and $\vartheta: X \rightarrow \mathcal{E}(\mathfrak{g}) \otimes \Omega_{X}( {\rm log} \, D)$ is a logarithmic Higgs field. Moreover, $(\mathcal{E},\vartheta)$ is \emph{$R$-semistable} (resp. \emph{$R$-stable}) if
\begin{itemize}
	\item for any maximal parabolic subgroup $\mathcal{P} \subseteq (\mathcal{G}_{\boldsymbol\alpha})_\eta$;
	\item for any antidominant character $\kappa: \mathcal{P} \rightarrow \mathbb{G}_{m,\eta}$ trivial on the center;
	\item for any $\vartheta$-compatible reduction of structure group $\varsigma: X_\eta \rightarrow \mathcal{E}_\eta / \mathcal{P}$,
\end{itemize}
we have
\begin{align*}
    par \deg L(\varsigma,\kappa) \geq 0 \text{ (resp. $> 0$)}.
\end{align*}
\end{defn}

\begin{prop}\label{prop_parah_equiv_G_Higgs_stab}
Logahoric $\mathcal{G}_{\boldsymbol\alpha}$-Higgs torsors $(\mathcal{E},\vartheta)$ on $X$ are in one-to-one correspondence with logarithmic $G$-Higgs bundles $(\mathscr{E},\theta)$ of type $\boldsymbol\rho$ on $\mathscr{X}$. Moreover, this correspondence preserves the stability conditions.
\end{prop}

\begin{proof}
By Proposition \ref{prop_parah_equiv_G}, parahoric $\mathcal{G}_{\boldsymbol\alpha}$-torsors are in one-to-one correspondence with $G$-bundles on $\mathscr{X}$ of type $\boldsymbol\rho$, and we have to prove such a correspondence also holds for Higgs fields. By Hartog's lemma, we can work on $X^{\rm big}$ and thus, it is enough to work on the associated tuple $(\mathcal{E}_0, \mathcal{E}_i, \Theta'_i)$ (resp. $(\mathscr{E}_0, \mathscr{E}_i, \Theta_i)$) of parahoric $\mathcal{G}_{\boldsymbol\alpha}$-torsors $\mathcal{E}$ on $X$ (resp. $G$-bundles on $\mathscr{X}$) by Remark \ref{rem_trans_func}. We follow the same notation as in \S\ref{subsect_G_stack} and work on $[\Spec \, \widetilde{R}_i / \mu_{d_i}]$ (resp. $\Spec \, R_i$) with local coordinate $w_i$ (resp. $z_i$).

We start with a Higgs field $\theta$ on $\mathscr{E}$. The restriction of the Higgs field $\theta_i$ on $\mathscr{E}_i$ is $\mu_{d_i}$-equivariant, i.e.
\begin{align*}
    \vartheta_i (\gamma_i w_i) = \rho(\gamma_i) \vartheta_i(w_i) \rho(\gamma_i)^{-1}.
\end{align*}
Define $\vartheta_i:= \Delta_i^{-1} \vartheta_i \Delta_i$, and the new element $\vartheta_i$ is $\mu_{d_i}$-invariant, i.e.
\begin{align*}
    \vartheta_i(\gamma_i w_i) = \vartheta_i(w_i).
\end{align*}
Thus, $\vartheta_i(z_i)$ is well-defined by taking $z_i = w_i^{d_i}$. With the same argument as in Proposition \ref{prop_parah_equiv_G}, the new element $\vartheta_i$ gives a section $\Spec \, R_i \rightarrow \mathcal{E}(\mathfrak{g}) \otimes \Omega_X( {\rm log} \, D)|_{\Spec \, R_i}$. Gluing $\vartheta_i$ via transition functions $\Theta'_i$, we obtain a Higgs field $X^{\rm big} \rightarrow \mathcal{E}^{\rm big}(\mathfrak{g}) \otimes \Omega_{X^{\rm big}}( {\rm log} \, D)$ and thus a Higgs field $\vartheta: X \rightarrow \mathcal{E}(\mathfrak{g}) \otimes \Omega_{X}( {\rm log} \, D)$ by Hartog's lemma. This construction also holds in the other direction by Remark \ref{rem_Higgs_field} and thus a one-to-one correspondence.

Based on the correspondence for Higgs fields given above and the correspondence for reductions of structure group given in Proposition \ref{prop_parah_equiv_G_stab}, there is a one-to-one correspondence between $\vartheta$-compatible generic reductions $\varsigma: X_\eta \rightarrow \mathcal{E}_\eta / \mathcal{P}$ and $\theta$-compatible reductions $\sigma: \mathscr{X}^{\rm big} \rightarrow (\mathscr{E}|_{\mathscr{X}^{\rm big}})/P$ on some big open substack $\mathscr{X}^{\rm big}$. Therefore, the correspondence preserves the stability conditions.
\end{proof}

Note that $\mathcal{E}(\mathfrak{g})$ is a locally free sheaf, of which fibers are isomorphic to $\mathfrak{g}$, and a logarhoric Higgs field is a section $
\varsigma: X \rightarrow \mathcal{E}(\mathfrak{g}) \otimes_{\mathcal{O}_X} \Omega_{X} ( {\rm log} \, D)$. Therefore, nilpotent logahoric Higgs torsors is defined in the same way as that for $G$-Higgs bundles in \S\ref{subsect_nil}. Then we introduce the following categories:
\begin{itemize}
    \item ${\rm HIG}_{p-1}(X_{\rm log}/k,\mathcal{G}_{\boldsymbol\alpha})$: the category of nilpotent logahoric $\mathcal{G}_{\boldsymbol\alpha}$-Higgs torsors on $X$ of exponent $\leq p-1$;
    \item ${\rm HIG}_{p-1}(\mathscr{X}_{\rm log}/k , G, \boldsymbol\rho)$: the category of nilpotent logarithmic $G$-Higgs bundles of type $\boldsymbol\rho$ on $\mathscr{X}$ of exponent $\leq p-1$.
\end{itemize}

As a direct result of Proposition \ref{prop_parah_equiv_G_Higgs_stab}, we have the desired equivalence of categories:

\begin{cor}\label{cor_parah_equiv_G_Higgs}
The following categories are equivalent:
\begin{itemize}
    \item ${\rm HIG}_{p-1}(\mathscr{X}_{\rm log}/k , G, \boldsymbol\rho)$ and ${\rm HIG}_{p-1}(X_{\rm log}/k,\mathcal{G}_{\boldsymbol\alpha})$;
    \item ${\rm HIG}^{(s)s}_{p-1}(\mathscr{X}_{\rm log}/k , G, \boldsymbol\rho)$ and ${\rm HIG}^{(s)s}_{p-1}(X_{\rm log}/k,\mathcal{G}_{\boldsymbol\alpha})$.
\end{itemize}
\end{cor}

Now we consider logarithmic connections. Chen--Zhu gave the definition of (integrable) connections on $\mathcal{G}$-torsors, where $\mathcal{G}$ is an arbitrary smooth affine group scheme over $X$ \cite[Appendix]{CZ15}. We follow their setup and define logarithmic integrable connections on parahoric torsors. Let $\mathcal{V}$ be a parahoric $\mathcal{G}_{\boldsymbol\alpha}$-torsor. A(n) \emph{(integrable) logarithmic connection} $\nabla'$ on $\mathcal{V}$ is a(n) (integrable) connection $\nabla': \mathcal{O}_{\mathcal{V}} \rightarrow \mathcal{O}_{\mathcal{V}} \otimes_{\mathcal{O}_X} \Omega_{X}({\rm log} \, D)$ compatible with the multiplication of $\mathcal{O}_{\mathcal{V}}$ such that the following diagram commutes
\begin{equation*}
\begin{tikzcd}
\mathcal{O}_{\mathcal{V}} \arrow[r] \arrow[d, "\nabla'"] & \mathcal{O}_{\mathcal{V}} \otimes_{\mathcal{O}_X} \mathcal{O}_{G \times X} \arrow[d,"\nabla' \otimes 1 + 1 \otimes \nabla_{\mathcal{G}_{\boldsymbol\alpha}}"] \\
\mathcal{O}_{\mathcal{V}} \otimes_{\mathcal{O}_X} \Omega_{X}({\rm log} \, D) \arrow[r] & (\mathcal{O}_{\mathcal{V}} \otimes_{\mathcal{O}_X} \mathcal{O}_{G \times X}) \otimes_{\mathcal{O}_X} \Omega_{X}({\rm log} \, D)
\end{tikzcd}
\end{equation*}
where $\nabla_{\mathcal{G}_{\boldsymbol\alpha}}$ is an integrable connection on $\mathcal{O}_{\mathcal{G}_{\boldsymbol\alpha}}$.

\begin{defn}
A \emph{logahoric $\mathcal{G}_{\boldsymbol\alpha}$-connection} is a pair $(\mathcal{V},\nabla')$, where $\mathcal{V}$ is a parahoric $\mathcal{G}_{\boldsymbol\alpha}$-torsor and $\nabla': \mathcal{O}_{\mathcal{V}} \rightarrow \mathcal{O}_{\mathcal{V}} \otimes \Omega_{X}({\rm log}\, D)$ is a logarithmic integrable connection.
\end{defn}

Given a logahoric $\mathcal{G}_{\boldsymbol\alpha}$-connection, denote by $\psi':=\psi(\nabla')$ the $p$-curvature of $\nabla'$. Then, the nilpotency of a logarithmic integrable connection $\nabla'$ is defined in the same way as \S\ref{subsect_nil}. Furthermore, we define the residue map
\begin{align*}
    {\rm Res}_{D_i} \nabla': \mathcal{O}_{\mathcal{V}} \otimes \mathcal{O}_{D_i} \rightarrow \mathcal{O}_{\mathcal{V}} \otimes \mathcal{O}_{D_i}
\end{align*}
as in \S\ref{subsect_log_G_Higgs_and_conn}. We introduce the following categories:
\begin{itemize}
    \item ${\rm MIC}_{p-1}(X_{\rm log}/k, \mathcal{G}_{\boldsymbol\alpha})$: the category of nilpotent logahoric $\mathcal{G}_{\boldsymbol\alpha}$-connections $(\mathcal{V},\nabla')$ of exponent $\leq p-1$ on $X$ such that
    \begin{enumerate}
        \item the semisimple part of the residue ${\rm Res}_{D_i} \nabla'$ is $\alpha_i$ for each $1 \leq i \leq s$, where we regard $\alpha_i$ as an element in the Lie algebra $\mathfrak{g}$;
        \item the nilpotent part of ${\rm Res}_{D_i} \nabla'$ is nilpotent of exponent $\leq p-1$ for every $1 \leq i \leq s$.
    \end{enumerate}
    \item ${\rm MIC}_{p-1}(\mathscr{X}_{\rm log}/k, G ,\boldsymbol\rho)$: the category of nilpotent flat $G$-bundles $(\mathscr{V},\nabla)$ of type $\boldsymbol\rho$ of exponent $\leq p-1$ on $\mathscr{X}$, of which the residue ${\rm Res}_{\widetilde{D}_i} \nabla$ is nilpotent of exponent $\leq p-1$ for every $1 \leq i \leq s$.
\end{itemize}
The extra condition that the semisimple part of the residue ${\rm Res}_{D_i} \nabla'$ is $\alpha_i$ for each $1 \leq i \leq s$ in the definition of the category ${\rm MIC}_{p-1}(X_{\rm log}/k, \mathcal{G}_{\boldsymbol\alpha})$ is called \emph{adjustness}, which is studied by the authors in \cite[\S 2]{KS20a}.

\begin{prop}\label{prop_parah_equiv_G_conn}
The categories ${\rm MIC}_{p-1}(X_{\rm log}/k, \mathcal{G}_{\boldsymbol\alpha})$ and ${\rm MIC}_{p-1}(\mathscr{X}_{\rm log}/k, G ,\boldsymbol\rho)$ are equivalent.
\end{prop}

\begin{proof}
    The discussion is similar to that of Higgs bundles (Proposition \ref{prop_parah_equiv_G_Higgs_stab}) and we only work on local charts $[\Spec \, \widetilde{R}_i / \mu_{d_i}]$ and $\Spec \, R_i$. The only thing we have to show is how the residues of connections change. Let $(\mathscr{V},\nabla)$ be a logarithmic flat $G$-bundle on $\mathscr{X}$ and denote by $\nabla_i$ the restriction of $\nabla$ to $\mathscr{V}_i$. Recall that the transition function $\Theta_i$ changes to $\Theta'_i = \Delta_i^{-1} \Theta_i$ (Proposition \ref{prop_parah_equiv_G}), which results in the change of Higgs fields $\theta_i= \Delta_i^{-1} \vartheta_i \Delta_i$ (Proposition \ref{prop_parah_equiv_G_Higgs_stab}). Therefore, we define a new connection $\nabla'_i (w_i):= \Delta_i(w_i) \circ \nabla_i(w_i)$, where the action is the gauge action. Regarding the element $\Delta_i(w_i)$ as $w_i^{d_i \alpha_i}$, we have
    \begin{align*}
        \nabla'_i(w_i) = d_i \alpha_i \frac{dw_i}{w_i} + \Delta_i^{-1}(w_i) \nabla_i(w_i) \Delta_i(w_i),
    \end{align*}
    which is $\mu_{d_i}$-invariant. Under the substitution $z_i=w_i^{d_i}$, we obtain a connection $\nabla'_i(z_i)$. Since the residue of $\nabla_i$ is nilpotent, the semisimple part of $\nabla'_i(z_i)$ is $\alpha_i$.
\end{proof}

\begin{defn}
Let $(\mathcal{V},\nabla')$ be a logahoric $\mathcal{G}_{\boldsymbol\alpha}$-connection. It is \emph{$R$-semistable} (resp. \emph{$R$-stable}) if
\begin{itemize}
	\item for any maximal parabolic subgroup $\mathcal{P} \subseteq (\mathcal{G}_{\boldsymbol\alpha})_\eta$;
	\item for any antidominant character $\kappa: \mathcal{P} \rightarrow \mathbb{G}_{m,\eta}$ trivial on the center;
	\item for any $\vartheta$-compatible reduction of structure group $\varsigma: X_\eta \rightarrow \mathcal{E}_\eta / \mathcal{P}$,
\end{itemize}
we have
\begin{align*}
    par \deg L(\varsigma,\kappa) \geq 0 \text{ (resp. $> 0$)}.
\end{align*}
\end{defn}

Adding the stability conditions, we have:
\begin{cor}\label{cor_parah_equiv_G_conn_stab}
The categories ${\rm MIC}^{(s)s}_{p-1}(X_{\rm log}/k, \mathcal{G}_{\boldsymbol\alpha})$ and ${\rm MIC}^{(s)s}_{p-1}(\mathscr{X}_{\rm log}/k, G ,\boldsymbol\rho)$ are equivalent.
\end{cor}

Denote by $\boldsymbol\rho$ the collection of corresponding representations of $\boldsymbol\alpha$. Let $\boldsymbol\rho^p$ be the collection of representations given in Remark \ref{rem_type} and denote by $\boldsymbol\alpha'$ the corresponding collection of weights. In conclusion, we obtain a logahoric version of the nonabelian Hodge correspondence:
\begin{thm}\label{thm_log_NAHC}
Suppose that $(X,D)$ is $W_2(k)$-liftable. The following categories are equivalent:
\begin{itemize}
    \item ${\rm HIG}_{p-1}(X'_{\rm log}/k,\mathcal{G}_{\boldsymbol\alpha})$ and ${\rm MIC}_{p-1}(X_{\rm log}/k, \mathcal{G}_{\boldsymbol\alpha'})$,
    \item ${\rm HIG}^{(s)s}_{p-1}(X'_{\rm log}/k,\mathcal{G}_{\boldsymbol\alpha})$ and ${\rm MIC}^{(s)s}_{p-1}(X_{\rm log}/k, \mathcal{G}_{\boldsymbol\alpha'})$.
\end{itemize}
\end{thm}

\begin{proof}
The proof of the first equivalence follows directly from the following diagram,
\begin{equation*}
\begin{tikzcd}
{\rm HIG}_{p-1}(\mathscr{X}'_{\rm log}/k , G, \boldsymbol\rho) \arrow[rr, "{\rm Theorem} \,\ref{thm_HIG_MIC_G_log_stack}", equal] \arrow[dd, "{\rm Corollary} \, \ref{cor_parah_equiv_G_Higgs}", equal] & & {\rm MIC}_{p-1}(\mathscr{X}_{\rm log}/k, G ,\boldsymbol\rho^p) \arrow[dd,"{\rm Proposition} \, \ref{prop_parah_equiv_G_conn}", equal] \\
& & \\
{\rm HIG}_{p-1}(X'_{\rm log}/k,\mathcal{G}_{\boldsymbol\alpha}) \arrow[rr, dotted, dash] & & {\rm MIC}_{p-1}(X_{\rm log}/k, \mathcal{G}_{\boldsymbol\alpha'})
\end{tikzcd}
\end{equation*}
and the second equivalence is given via the following diagram.
\begin{equation*}
\begin{tikzcd}
{\rm HIG}^{(s)s}_{p-1}(\mathscr{X}'_{\rm log}/k , G, \boldsymbol\rho) \arrow[rr, "{\rm Theorem} \,\ref{thm_HIG_MIC_G_stab_stack}", equal] \arrow[dd, "{\rm Corollary} \, \ref{cor_parah_equiv_G_Higgs}", equal] & & {\rm MIC}^{(s)s}_{p-1}(\mathscr{X}_{\rm log}/k, G ,\boldsymbol\rho^p) \arrow[dd,"{\rm Corollary} \, \ref{cor_parah_equiv_G_conn_stab}", equal] \\
& & \\
{\rm HIG}^{(s)s}_{p-1}(X'_{\rm log}/k,\mathcal{G}_{\boldsymbol\alpha}) \arrow[rr, dotted, dash] & & {\rm MIC}^{(s)s}_{p-1}(X_{\rm log}/k, \mathcal{G}_{\boldsymbol\alpha'})
\end{tikzcd}
\end{equation*}
The only thing we have to prove is that if $(X,D)$ is $W_2(k)$-liftable, the corresponding pairs $(\mathscr{X},\widetilde{D})$ is also $W_2(k)$-liftable, where $\mathscr{X}$ is the root stack determined by the data $(X,D,\boldsymbol{d})$. By the Kawamata covering lemma \cite[Theorem 17]{Kaw81}, there exists a covering $\pi': Y \rightarrow X$ such that $Y$ is smooth and $D':=(\pi'^* D)_{red}$ is a reduced normal crossing divisor on $Y$. Moreover, we can assume that the induced morphism $Y \rightarrow \mathscr{X}$ is surjective and \'etale \cite[Lemma 4.2]{Sim11}. Note that although Kawamata and Simpson stated the results only in characteristic zero, their arguments hold in our case due to the assumption that all integers in $\boldsymbol{d}$ are coprime to $p$. Then it is enough to show that the pair $(Y, D')$ is $W_2(k)$-liftable. This property has been proved in \cite[Theorem 1.1]{XW13} and we also refer the reader to \cite[Notation 2.16]{KS20a} for the case of curves.
\end{proof}

\begin{rem}
When $X$ is a curve and $G= {\rm GL}_n$, the correspondence given in Theorem \ref{thm_log_NAHC} has been studied in \cite[Theorem 2.10]{KS20a}, where the author used parabolic Higgs bundles and parabolic flat bundles to prove this result.
\end{rem}

\vspace{2mm}
{\bf Acknowledgement.} The first and third named authors are partially supported by the National Key R and D Program of China 2020YFA0713100, CAS Project for Young Scientists in Basic Research Grant No. YSBR-032. The second named author is supported by Guangdong Basic and Applied Basic Research Foundation 2024A1515011583 and National Key R and D Program of China 2022YFA1006600.

\bibliographystyle{amsalpha} %ref_bib
\bibliography{ref_pnahc}

\providecommand{\bysame}{\leavevmode\hbox to3em{\hrulefill}\thinspace}
\providecommand{\MR}{\relax\ifhmode\unskip\space\fi MR }
% \MRhref is called by the amsart/book/proc definition of \MR.
\providecommand{\MRhref}[2]{%
  \href{http://www.ams.org/mathscinet-getitem?mr=#1}{#2}
}
\providecommand{\href}[2]{#2}
\begin{thebibliography}{Ram96b}

\bibitem[AB01]{AB01}
B.~Anchouche and I.~Biswas, \emph{Einstein-{H}ermitian connections on
  polystable principal bundles over a compact {K}\"{a}hler manifold}, Amer. J.
  Math. \textbf{123} (2001), no.~2, 207--228. \MR{1828221}

\bibitem[AOV08]{AOV08}
D.~Abramovich, M.~Olsson, and A.~Vistoli, \emph{Tame stacks in positive
  characteristic}, Ann. Inst. Fourier (Grenoble) \textbf{58} (2008), no.~4,
  1057--1091.

\bibitem[BDP17]{BDP17}
V.~Balaji, P.~Deligne, and A.~J. Parameswaran, \emph{On complete reducibility
  in characteristic {$p$}}, \'Epijournal G\'eom. Alg\'ebrique \textbf{1}
  (2017), Art. 3, 27, With an appendix by Zhiwei Yun. \MR{3743106}

\bibitem[Bis10]{Bis10}
I.~Biswas, \emph{The {A}tiyah bundle and connections on a principal bundle},
  Proc. Indian Acad. Sci. Math. Sci. \textbf{120} (2010), no.~3, 299--316.

\bibitem[BLR90]{BLR90}
S.~Bosch, W.~L\"{u}tkebohmert, and M.~Raynaud, \emph{N\'{e}ron models},
  Ergebnisse der Mathematik und ihrer Grenzgebiete (3) [Results in Mathematics
  and Related Areas (3)], vol.~21, Springer-Verlag, Berlin, 1990.

\bibitem[Bor07]{Bor07}
N.~Borne, \emph{Fibr\'{e}s paraboliques et champ des racines}, Int. Math. Res.
  Not. IMRN (2007), no.~16, Art. ID rnm049, 38.

\bibitem[BP22]{BP22}
V.~Balaji and Y.~Pandey, \emph{On bruhat--tits theory over a higher dimensional
  base}, arXiv:2203.0431v2 (2022).

\bibitem[BS15]{BS15}
V.~Balaji and C.~S. Seshadri, \emph{Moduli of parahoric {G}-torsors on a
  compact {R}iemann surface}, J. Algebraic Geom. \textbf{24} (2015), no.~1,
  1--49.

\bibitem[Cad07]{Cad07}
C.~Cadman, \emph{Using stacks to impose tangency conditions on curves}, Amer.
  J. Math. \textbf{129} (2007), no.~2, 405--427.

\bibitem[CN20]{CN20}
T.~H. Chen and B.~C. Ng\^{o}, \emph{On the {H}itchin morphism for
  higher-dimensional varieties}, Duke Math. J. \textbf{169} (2020), no.~10,
  1971--2004.

\bibitem[CS21]{CS21}
B.~Collier and A.~Sanders, \emph{({G},{P})-opers and global {S}lodowy slices},
  Adv. Math. \textbf{377} (2021), Paper No. 107490, 43.

\bibitem[CZ15]{CZ15}
T.~H. Chen and X.~Zhu, \emph{Non-abelian {H}odge theory for algebraic curves in
  characteristic {$p$}}, Geom. Funct. Anal. \textbf{25} (2015), no.~6,
  1706--1733.

\bibitem[Dam24]{Dam24}
C.~Damiolini, \emph{On equivariant bundles and their moduli spaces}, C. R.
  Math. Acad. Sci. Paris \textbf{362} (2024), 55--62.

\bibitem[DI87]{DI87}
P.~Deligne and L.~Illusie, \emph{Rel\`evements modulo {$p^2$} et
  d\'{e}composition du complexe de de {R}ham}, Invent. Math. \textbf{89}
  (1987), no.~2, 247--270.

\bibitem[GLSS08]{GLSS08}
T.~L. G\'omez, A.~Langer, A.~H.~W. Schmitt, and I.~Sols, \emph{Moduli spaces
  for principal bundles in arbitrary characteristic}, Adv. Math. \textbf{219}
  (2008), no.~4, 1177--1245.

\bibitem[GW20]{GW20}
U.~G\"ortz and T.~Wedhorn, \emph{Algebraic geometry {I}. {S}chemes---with
  examples and exercises}, second ed., Springer Studium Mathematik---Master,
  Springer Spektrum, Wiesbaden, [2020] \copyright 2020.

\bibitem[Hei10]{Hei10}
J.~Heinloth, \emph{Uniformization of {$G$}-bundles}, Math. Ann. \textbf{347}
  (2010), no.~3, 499--528.

\bibitem[Hei17]{Hei17}
\bysame, \emph{Hilbert-{M}umford stability on algebraic stacks and applications
  to {$G$}-bundles on curves}, \'{E}pijournal G\'{e}om. Alg\'{e}brique
  \textbf{1} (2017), Art. 11, 37.

\bibitem[Her13]{Her13}
S.~Herpel, \emph{On the smoothness of centralizers in reductive groups}, Trans.
  Amer. Math. Soc. \textbf{365} (2013), no.~7, 3753--3774.

\bibitem[Hum75]{Hum75}
James~E. Humphreys, \emph{Linear algebraic groups}, Graduate Texts in
  Mathematics, vol. No. 21, Springer-Verlag, New York-Heidelberg, 1975.
  \MR{396773}

\bibitem[HZ23]{HZ23}
A.~Herrero and S.~Zhang, \emph{Meromorphic {H}odge moduli spaces for reductive
  groups in arbitrary characteristic}, arXiv:2307.16755 (2023).

\bibitem[Kat70]{Kat70}
N.~Katz, \emph{Nilpotent connections and the monodromy theorem: {A}pplications
  of a result of {T}urrittin}, Inst. Hautes \'{E}tudes Sci. Publ. Math. (1970),
  no.~39, 175--232.

\bibitem[Kaw81]{Kaw81}
Y.~Kawamata, \emph{Characterization of abelian varieties}, Compositio Math.
  \textbf{43} (1981), no.~2, 253--276.

\bibitem[Kos93]{Kos93}
T.~Kosir, \emph{On the structure of commutative matrices}, Linear Algebra Appl.
  \textbf{187} (1993), 163--182.

\bibitem[KS20]{KS20a}
R.~Krishnamoorthy and M.~Sheng, \emph{Periodic de rham bundles ov curves},
  arXiv:2011.03268 (2020).

\bibitem[KSZ24]{KSZ23}
G.~Kydonakis, H.~Sun, and L.~Zhao, \emph{Logahoric {H}iggs torsors for a
  complex reductive group}, Math. Ann. \textbf{388} (2024), no.~3, 3183--3228.

\bibitem[LS24]{LS21}
M.~Li and H.~Sun, \emph{Tame parahoric nonabelian {H}odge correspondence in
  positive characteristic over algebraic curves}, Selecta Math. (N.S.)
  \textbf{30} (2024), no.~4, Paper No. 60.

\bibitem[LSYZ19]{LSYZ19}
G.~Lan, M.~Sheng, Y.~Yang, and K.~Zuo, \emph{Uniformization of {$p$}-adic
  curves via {H}iggs--de {R}ham flows}, J. Reine Angew. Math. \textbf{747}
  (2019), 63--108.

\bibitem[LSZ15]{LSZ15}
G.~Lan, M.~Sheng, and K.~Zuo, \emph{Nonabelian {H}odge theory in positive
  characteristic via exponential twisting}, Math. Res. Lett. \textbf{22}
  (2015), no.~3, 859--879.

\bibitem[McN02]{Mcn02}
G.~J. McNinch, \emph{Abelian unipotent subgroups of reductive groups}, J. Pure
  Appl. Algebra \textbf{167} (2002), no.~2-3, 269--300.

\bibitem[Mil17]{Mil17}
J.~S. Milne, \emph{Algebraic groups}, Cambridge Studies in Advanced
  Mathematics, vol. 170, Cambridge University Press, Cambridge, 2017, The
  theory of group schemes of finite type over a field.

\bibitem[MO05]{MO05}
K.~Matsuki and M.~Olsson, \emph{Kawamata-{V}iehweg vanishing as {K}odaira
  vanishing for stacks}, Math. Res. Lett. \textbf{12} (2005), no.~2-3,
  207--217.

\bibitem[OV07]{OV07}
A.~Ogus and V.~Vologodsky, \emph{Nonabelian {H}odge theory in characteristic
  {$p$}}, Publ. Math. Inst. Hautes \'{E}tudes Sci. (2007), no.~106, 1--138.

\bibitem[Ram75]{Ram75}
A.~Ramanathan, \emph{Stable principal bundles on a compact {R}iemann surface},
  Math. Ann. \textbf{213} (1975), 129--152.

\bibitem[Ram79]{Ram78}
\bysame, \emph{Moduli for principal bundles}, Algebraic geometry ({P}roc.
  {S}ummer {M}eeting, {U}niv. {C}openhagen, {C}openhagen, 1978), Lecture Notes
  in Math., vol. 732, Springer, Berlin, 1979, pp.~527--533.

\bibitem[Ram96a]{Ram96a}
\bysame, \emph{Moduli for principal bundles over algebraic curves. {I}}, Proc.
  Indian Acad. Sci. Math. Sci. \textbf{106} (1996), no.~3, 301--328.

\bibitem[Ram96b]{Ram96b}
\bysame, \emph{Moduli for principal bundles over algebraic curves. {II}}, Proc.
  Indian Acad. Sci. Math. Sci. \textbf{106} (1996), no.~4, 421--449.

\bibitem[Sch05]{Sch05}
D.~Schepler, \emph{Logarithmic nonabelian {H}odge theory in characteristic p},
  ProQuest LLC, Ann Arbor, MI, 2005, Thesis (Ph.D.)--University of California,
  Berkeley.

\bibitem[Sei00]{Sei00}
G.~M. Seitz, \emph{Unipotent elements, tilting modules, and saturation},
  Invent. Math. \textbf{141} (2000), no.~3, 467--502.

\bibitem[Ser94]{Ser94}
J.-P. Serre, \emph{Sur la semi-simplicit\'{e} des produits tensoriels de
  repr\'{e}sentations de groupes}, Invent. Math. \textbf{116} (1994), no.~1-3,
  513--530.

\bibitem[She24]{She21}
M.~Sheng, \emph{Twisted funcotiality in nonabelian hodge theory in positive
  characteristic}, to appear in Pure Appl. Math. Q. (2024).

\bibitem[Sim90]{Sim90}
C.~T. Simpson, \emph{Harmonic bundles on noncompact curves}, J. Amer. Math.
  Soc. \textbf{3} (1990), no.~3, 713--770.

\bibitem[Sim11]{Sim11}
\bysame, \emph{Local systems on proper algebraic {$V$}-manifolds}, Pure Appl.
  Math. Q. \textbf{7} (2011), no.~4, Special Issue: In memory of Eckart
  Viehweg, 1675--1759.

\bibitem[Sob15]{Sob15b}
P.~Sobaje, \emph{Exponentiation of commuting nilpotent varieties}, J. Pure
  Appl. Algebra \textbf{219} (2015), no.~6, 2206--2217.

\bibitem[XW13]{XW13}
Q.~Xie and J.~Wu, \emph{Strongly liftable schemes and the {K}awamata-{V}iehweg
  vanishing in positive characteristic {III}}, J. Algebra \textbf{395} (2013),
  12--23.

\bibitem[Yan22]{Yan22}
M.~Yang, \emph{A comparison of generalized opers and {$(G,P)$}-opers}, Indian
  J. Pure Appl. Math. \textbf{53} (2022), no.~3, 760--773.

\bibitem[Yun11]{Yun11}
Z.~Yun, \emph{Global {S}pringer theory}, Adv. Math. \textbf{228} (2011), no.~1,
  266--328.

\end{thebibliography}

\bigskip

\noindent\small{\textsc{Yau Mathematical Science Center, Tsinghua University, \\
Beijing, 100084, China}\\
\noindent\small{\textsc{Yanqi Lake Beijing Institute of Mathematical Sciences and Applications,\\
Beijing, 101408, China}\\
\emph{E-mail address}:  \texttt{msheng@mail.tsinghua.edu.cn}

\bigskip
\noindent\small{\textsc{Department of Mathematics, South China University of Technology, \\
Guangzhou, 510641, China}\\
\emph{E-mail address}:  \texttt{hsun71275@scut.edu.cn}

\bigskip
\noindent\small{\textsc{Yau Mathematical Science Center, Tsinghua University, \\
Beijing, 100084, China}\\
\emph{E-mail address}:  \texttt{jianpw@mail.tsinghua.edu.cn}

\end{document}